\documentclass[12pt,a4paper,leqno]{article}
\usepackage{amsmath,amsfonts,amsthm,amssymb}
\usepackage{mathrsfs,array,float}
\theoremstyle{plain}
\newtheorem{theorem}{Theorem}[section]
\newtheorem{proposition}[theorem]{Proposition}
\newtheorem{lemma}[theorem]{Lemma}
\newtheorem{corollary}[theorem]{Corollary}
\newtheorem{claim}[theorem]{Claim}

\theoremstyle{definition}
\newtheorem{definition}[theorem]{Definition}
\newtheorem{example}[theorem]{Example}
\theoremstyle{remark}
\newtheorem{remark}[theorem]{Remark}
\numberwithin{equation}{section}

\renewcommand{\eqref}[1]{(\ref{#1})}
\newcommand{\field}[1]{\mathbb{#1}}
\newcommand{\C}{\field{C}}

\newcommand{\id}{{\mathrm{id}}}
\newcommand{\setH}{{Set_H}}
\newcommand{\quivH}{{Quiv_H}}
\newcommand{\shikaku}[2]{{\underset{#2}{\overset{#1}{\square}}}}
\begin{document}
\title{Dynamical Reflection Maps}
\author{Ryosuke ASHIKAGA\thanks{Department of Mathematics,
Graduate School of Science, 
Hokkaido University, Sapporo 0600810, JAPAN; 
ashikaga@math.sci.hokudai.ac.jp}
\ 
and Youichi SHIBUKAWA\thanks{Department of Mathematics,
Faculty of Science, 
Hokkaido University, Sapporo 0600810, JAPAN; 
shibu@math.sci.hokudai.ac.jp}}
\maketitle
\begin{abstract}
In this paper,
by making use of category theory,
we construct dynamical reflection maps,
solutions to a version of the reflection equation
associated with suitable dynamical Yang-Baxter maps,
set-theoretic solutions to the braid relation that is
equivalent to a version of the quantum Yang-Baxter equation.
Quiver-theoretic solutions
to the reflection equation
are also discussed.
\end{abstract}
\section{Introduction}
The quantum Yang-Baxter equation \cite{baxter1,baxter2,yang}
has been studied
intensively
in mathematics and physics.
Much research for finding solutions to this equation
gave birth to the quantum groups
\cite{drinfeld2,jimbo},
examples of non-commutative and non-cocommutative Hopf algebras.

The quantum Yang-Baxter equation 
is defined on the tensor product $V\otimes V\otimes V$ 
of a vector space $V$.
Instead of the tensor products,
Drinfel'd \cite[Section 9]{drinfeld} proposed to study
the quantum Yang-Baxter equation
on the Cartesian product $S\times S\times S$
of a set $S$,
whose solution is
called a Yang-Baxter map
\cite{etingof,lu,veselov,weinstein}.
This Yang-Baxter map
plays an important role in 
discrete integrable systems
\cite{adler}.

The quantum Yang-Baxter equation 
is exactly 
the braid relation \eqref{eq:braidrel}
in the tensor category consisting of 
$\C$-vector spaces,
and,
from this viewpoint,
the Yang-Baxter map is regarded as a solution to the braid relation
in the tensor category consisting of sets.
We can consider
the braid relation
in another tensor category $\setH$
\cite{shibukawa,shibu3},
which is embedded into the
tensor category $\quivH$ consisting of
quivers over  a nonempty set $H$
(For $\setH$ and $\quivH$, see Sections \ref{section:setH}
and \ref{section:quivertheoretic} respectively).

Suitable ternary operations
\cite{kamiya1,kamiya2,shibu2},
dynamical braces
\cite{matsumoto3},
and
braided semigroups
\cite{matsumoto}
produce dynamical Yang-Baxter maps
\cite{shibu1,shibu2},
solutions to the braid relation in $\setH$.
This dynamical Yang-Baxter map is useful in
discrete integrable systems
\cite{kouloukas}.
Moreover,
by means of the dynamical Yang-Baxter map,
we can construct
quiver-theoretic solutions to the braid relation
\cite{matsumoto2}
and
Hopf algebroids
\cite{otsuto,shibu4},
which give birth to rigid tensor categories
consisting of finite-dimensional dynamical representations \cite{shibu4}.

We will try to investigate
so-called dynamical reflection equation algebras
\cite{kulish}
associated with the dynamical Yang-Baxter maps.
As a first step, 
in order to confirm that this dynamical reflection equation algebra is rich in representations,
this paper focuses on the reflection equation in
\cite[Eq.(10)]{donin}.
A set-theoretic version of this reflection equation is considered in
\cite{decommer},
which clarifies a way to construct reflection maps, 
(set-theoretic) solutions to this reflection equation,
by means of actions of skew left braces
(For skew left braces, see Definition \ref{def:skewleftbrace}).

The aim of this paper is to construct
dynamical reflection maps,
solutions to the reflection equation
\eqref{eq:RE}
in the tensor category $\setH$,
associated with
the dynamical Yang-Baxter map.

Let $L$ be a left quasigroup with a unit
(Definition \ref{def:leftquasig}),
and let $G$ be a group isomorphic to $L$ as sets.
We set $H=L$.
This left quasigroup $L$ with a unit can give birth to 
a monoid $L=(L, m, \eta)$ in $\setH$
(See Definition \ref{def:monoid}).
Let
$\sigma: L\otimes L\to L\otimes L$
be a dynamical Yang-Baxter map \eqref{eq:sigma}
defined by
the ternary operation $\mu_1^G$
\eqref{eq:mu1G}
related to the group $G$.
By means of a left $(L, m, \eta)$-module $(X, m_X)$ in $\setH$
and
a family of homomorphisms of the group $G$,
we can produce
a dynamical reflection map 
$k: L\otimes X\to L\otimes X$ 
associated with
the dynamical Yang-Baxter map
$\sigma: L\otimes L\to L\otimes L$
(See Theorem \ref{theorem:grouphomo1}
and the end of Section \ref{section:construction}).
As an application,
we construct quiver-theoretic solutions to the reflection equation
from the dynamical reflection maps
(Proposition \ref{prop:quiverrefl}).

The organization of this paper is as follows. 

Section \ref{section:setH}
contains a brief summary of
the tensor category $Set_H$,
in which we will consider the braid relation and the reflection equation. 
In Section \ref{section:reflection},
we present a way to
construct dynamical reflection maps,
by generalizing the method in \cite{decommer}
from the viewpoint of category theory.

We introduce the notions of monoids in Section \ref{section:braided}
and their left modules in Section \ref{section:leftmod},
which play a key role in constructing dynamical reflection maps
in this paper.
In particular,
braided monoids
in Definition \ref{def:braided}
and
braid-commuting pairs of left modules
in Definition \ref{def:braid-commute}
are essential.
The relations in Section \ref{section:reflection}
imply
these notions
(See Remarks \ref{rem:braidcommmYmYtriv}
and \ref{rem:braidcommmYmYsigma} also).
The tensor category $\setH$ has an advantage
in
constructing
desirable braided monoids
(See
Proposition \ref{prop:braidedmonoid}).

Our main result,
Theorem \ref{theorem:grouphomo1},
is stated in
Section \ref{section:construction}
and
proved in Section \ref{section:proof}.
By means of this theorem,
suitable left modules 
and
group homomorphisms can
produce the dynamical reflection maps.

In Section \ref{section:examples},
we provide several examples.
The end of this section
deals with reflection maps also.
In Section \ref{section:quivertheoretic}, the final section,
we prove that the dynamical reflection maps
give birth to quiver-theoretic solutions to the reflection equation,
following \cite{matsumoto2}.
\section{Tensor category $\setH$}
\label{section:setH}
This section contains a brief summary of 
the tensor category 
$\setH$ 
\cite{shibukawa,shibu3},
which
plays an important role in this paper.
About the tensor categories,
we follow the definition and the notation of \cite[Chapter XI]{kassel}
throughout the paper.

\begin{definition}\label{def:tensorcat}
A tensor category $(\mathcal{C}, \otimes, I, a, l, r)$ is a category
$\mathcal{C}$,
together with:
\begin{enumerate}
\item
a functor $\otimes: \mathcal{C}\times\mathcal{C}\to\mathcal{C}$
called a tensor product;
\item
a natural isomorphism
$a: \otimes(\otimes\times\id)\to
\otimes(\id\times\otimes)$
called an associativity constraint;
\item
an object $I$ called a unit;
\item
a natural isomorphim
$l: \otimes(I\times\id)\to\id$
called a left unit constraint with respect to the unit $I$;
\item
$r: \otimes(\id\times I)\to\id$
called a right unit constraint with respect to the unit $I$,
\end{enumerate}
satisfying
the pentagon axiom and the triangle axiom
(See \cite[Definition XI.2.1]{kassel}):
for $U, V, W, X\in \mathcal{C}$;
\begin{align}
&(1_W\otimes a_{VWX})\circ a_{UV\otimes WX}
\circ(a_{UVW}\otimes 1_X)
=a_{UVW\otimes X}\circ a_{U\otimes VWX};
\label{eq:pentagon}
\\
&(1_V\otimes l_W)\circ a_{VIW}=r_V\otimes 1_W.
\label{eq:triangle}
\end{align}
\end{definition}
We now recall the definition of $\setH$.
Let 
$H$
be a nonempty set.
$(X, \cdot_{X})$
is an object of the category
$\setH$,
iff
$X$
is a set and
$\cdot_{X}$
is a map from
$H \times X$
to
$H$.
For
$(X, \cdot_{X}), (Y, \cdot_{Y})\in\setH$,
we call 
$f: (X, \cdot_{X}) \to (Y, \cdot_{Y})$
a morphism of
$\setH$, 
iff 
$f$
is a map from 
$H$
to
$Map(X, Y)$
satisfying
\[
\lambda \cdot_{Y} f (\lambda)(x) = \lambda \cdot_{X} x
\quad(\forall\lambda \in H, \forall x \in X).
\] 
Here,
we denote by
$Map(X, Y)$
the set of all maps from $X$ to $Y$.
For morphisms $f: (X, \cdot_{X}) \to (Y, \cdot_{Y})$
and
$g: (Y, \cdot_{Y}) \to (Z, \cdot_{Z})$ of $\setH$,
we define the composition
$g \circ f: (X, \cdot_{X}) \to (Z, \cdot_{Z})$ by 
$(g \circ f) (\lambda) = g (\lambda) \circ f (\lambda)$
$(\lambda \in H )$.

Then $\setH$ is a category with the following identity $1_{(X, \cdot_X)}:
(X, \cdot_X)\to(X, \cdot_X)$:
For 
$(X, \cdot_X)\in\setH$,
the map
$1_{(X, \cdot_X)}: H \to Map(X, X)$
is defined by
$1_{(X, \cdot_X)}(\lambda)(x) = x$
$(\lambda \in H, x \in X)$.

The next task is to show that
$\setH$ is a tensor category.
Let
$X=(X, \cdot_X)$
and
$Y=(Y, \cdot_Y)$
be objects in
$\setH$.   
We define the tensor product
$(X \otimes Y, \cdot_{X \otimes Y})$
by:
\[
X \otimes Y = (X \times Y, \cdot_{X \otimes Y});
\lambda \cdot_{X \otimes Y} (x, y) = (\lambda \cdot_{X} x) \cdot_{Y} y
\quad(\forall
\lambda \in H, \forall(x, y) \in X\times Y).
\]
This
$X \otimes Y$ is an object of
$\setH$.  

The tensor product 
$f \otimes g: (X, \cdot_X) \otimes (X', \cdot_{X'}) \to 
(Y, \cdot_Y) \otimes (Y', \cdot_{Y'})$
of two morphisms
$f: (X, \cdot_X)\to (Y, \cdot_Y)$ and $g: (X', \cdot_{X'})\to (Y', \cdot_{Y'})$
is defined by
\[
(f \otimes g) (\lambda) (x, y) = (f (\lambda) (x), g (\lambda \cdot_X x) (y))
\quad(\lambda \in H, (x, y) \in X\times Y).
\]
It is easily seen that 
$f \otimes g$
is a morphism of
$\setH$. 

The associativity constraint 
$a: \otimes (\otimes \times \id) \to \otimes (\id \times \otimes)$
is given as follows: for
$X, Y, Z\in\setH$,
\[
a_{X, Y, Z}(\lambda) ((x, y), z) = (x, (y, z)) \ \ ( \lambda \in H, ((x, y), z) \in
((X \times Y) \times Z)).
\]

Let
$I = \{ \bullet \}$
denote a set of one element $\bullet$, together with the map 
$\cdot_{I}: H \times I \to H$
defined by
$\lambda \cdot_{I} \bullet = \lambda$
$(\lambda\in H)$.
This
$I=(I, \cdot_{I})$
is an object in
$\setH$,
which is called
a unit
of
$\setH$. 

The left and right unit constraints
$l: \otimes(I\times \id)\to\id$
and
$r: \otimes(\id\times I)\to\id$
with respect to the unit
$I$
are given by:
\[
l_{X}(\lambda)(\bullet, x) = x,
\quad
r_{X}(\lambda)(x, \bullet) = x
\quad
(\lambda\in H, x \in X)
\]
for every
$X\in\setH$.

We can check that
$(\setH, \otimes, I, a, l, r)$
is a tensor category.
\section{Reflection equation}
\label{section:reflection}
In this section, we establish a way to construct
solutions to the reflection equation
in tensor categories with suitable properties.
Finally we will present a sufficient condition
for one of these properties
in the case of the tensor category $\setH$.

Let $\mathcal{C}$ be a tensor category
(Definition \ref{def:tensorcat}).
\begin{proposition}\label{prop:k1}
Let $A, X\in\mathcal{C}$,
and
write
$Y=A\otimes X$.
Suppose that
morphisms
$m: A\otimes A\to A, \eta: I\to A,
m_X: A\otimes X\to X,
m_Y: A\otimes Y\to Y$,
and $\sigma: A\otimes A\to A\otimes A$
of $\mathcal{C}$
satisfy
\begin{align}
&m(1_A\otimes\eta)r_A^{-1}=1_A,
\label{eq:monoid2-2}
\\
&m_X(\eta\otimes 1_X)=l_X,
\label{eq:leftmod2}
\\
&m_Y(m\otimes 1_Y)=m_Y(1_A\otimes m_Y)a_{AAY},
\label{eq:leftmodmY}
\\
&m_Y(1_A\otimes m_Y^\mathrm{triv})=
m_Y^\mathrm{triv}(1_A\otimes m_Y)a_{AAY}(\sigma\otimes 1_Y)a_{AAY}^{-1},
\label{eq:braidcommmYmYtriv}
\\
&m_Xm_Y=m_X(1_A\otimes m_X).
\label{eq:mXmY}
\end{align}
Here,
\begin{equation}\label{eq:mYtriv}
m_Y^\mathrm{triv}:=(m\otimes 1_X)a_{AAX}^{-1}: A\otimes Y\to Y.
\end{equation}
Then
the morphism
\begin{equation}\label{eq:k}
k=m_Y(1_A\otimes ((\eta\otimes 1_X)l_X^{-1})): A\otimes X\to A\otimes X
\end{equation}
of $\mathcal{C}$ enjoys\/$:$
\begin{align}
&
m_Xk=m_X;
\label{eq:mXk3}
\\
\label{eq:mXk1}
&k(m\otimes 1_X)
\\
&\quad=(m\otimes 1_X)a_{AAX}^{-1}(1_A\otimes k)a_{AAX}
(\sigma\otimes 1_X)a_{AAX}^{-1}(1_A\otimes k)a_{AAX}.
\notag
\end{align}
\end{proposition}
\begin{proof}
We can easily prove \eqref{eq:mXk3}
by taking account of \eqref{eq:leftmod2},
\eqref{eq:mXmY},
and \eqref{eq:k}.

It follows from \eqref{eq:monoid2-2},
\eqref{eq:braidcommmYmYtriv}, \eqref{eq:mYtriv}, and \eqref{eq:k} that
\begin{equation}\label{eq:mYk1}
m_Y=(m\otimes 1_X)a_{AAX}^{-1}(1_A\otimes k)a_{AAX}(\sigma\otimes 1_X)
a_{AAX}^{-1}.
\end{equation}
In fact,
the right-hand-side of \eqref{eq:mYk1}
is
\[
m_Y^\mathrm{triv}(1_A\otimes m_Y)a_{AAY}(\sigma\otimes 1_{Y})
(1_{A\otimes A}\otimes(\eta\otimes 1_X))(1_{A\otimes A}\otimes l_X^{-1})
a_{AAX}^{-1}
\]
by \eqref{eq:mYtriv} and \eqref{eq:k}.
\eqref{eq:braidcommmYmYtriv} yields that
the right-hand-side of the above equation is
$m_Y(1_A\otimes m_Y^\mathrm{triv})a_{AAY}
(1_{A\otimes A}\otimes(\eta\otimes 1_X))(1_{A\otimes A}\otimes l_X^{-1})
a_{AAX}^{-1}$,
which is exactly $m_Y$ on account of 
\eqref{eq:triangle},
\eqref{eq:monoid2-2},
and
\eqref{eq:mYtriv}.

It is clear from \eqref{eq:leftmodmY} and \eqref{eq:k}
that
\begin{equation}\label{eq:mYk2}
k(m\otimes 1_X)=m_Y(1_A\otimes k)a_{AAX}.
\end{equation}
By combining \eqref{eq:mYk1} with \eqref{eq:mYk2},
we immediately get \eqref{eq:mXk1}.
\end{proof}
\begin{proposition}\label{prop:k2}
Let $A, X\in\mathcal{C}$,
and
set
$Y=A\otimes X$.
We assume that
morphisms
$m: A\otimes A\to A, \eta: I\to A,
m_X: A\otimes X\to X,
m_Y: A\otimes Y\to Y$,
and $\sigma: A\otimes A\to A\otimes A$
of $\mathcal{C}$
satisfy $\eqref{eq:braidcommmYmYtriv}$
and
\begin{align}
&m(\eta\otimes 1_A)l_A^{-1}=1_A=
m(1_A\otimes\eta)r_A^{-1},
\label{eq:monoid2}
\\
&(m\otimes 1_A)a_{AAA}^{-1}(1_A\otimes\sigma)a_{AAA}
(\sigma\otimes 1_A)=\sigma(1_A\otimes m)a_{AAA},
\label{eq:braidedmonoid2}
\\
&\sigma(1_A\otimes\eta)=(\eta\otimes 1_A)l_A^{-1}r_A,
\label{eq:braidedmonoid4}
\\
&a_{AAA}(\sigma\otimes 1_A)a_{AAA}^{-1}(1_A\otimes\sigma)
a_{AAA}(\sigma\otimes 1_A)
\label{eq:braidrel}\\
&\quad=
(1_A\otimes\sigma)a_{AAA}(\sigma\otimes 1_A)a_{AAA}^{-1}
(1_A\otimes\sigma)a_{AAA},
\notag
\\
&m_Y(1_A\otimes m_Y^\sigma)=m_Y^\sigma(1_A\otimes m_Y)a_{AAY}(\sigma\otimes 1_Y)
a_{AAY}^{-1}.
\label{eq:braidcommmYmYsigma}
\end{align}
Here,
\begin{equation}\label{eq:mYsigma}
m_Y^\sigma=(1_A\otimes m_X)a_{AAX}(\sigma\otimes 1_X)a_{AAX}^{-1}:
A\otimes Y\to Y.
\end{equation}
If $\sigma: A\otimes A\to A\otimes A$ is an isomorphism,
then the morphism
$k: A\otimes X\to A\otimes X$
$\eqref{eq:k}$ enjoys
\begin{align}\label{eq:mXk2}
&k(1_A\otimes m_X)
\\
&\quad=(1_A\otimes m_X)a_{AAX}(\sigma\otimes 1_X)a_{AAX}^{-1}
(1_A\otimes k)a_{AAX}(\sigma\otimes 1_X)a_{AAX}^{-1}.
\notag
\end{align}
\end{proposition}
\begin{remark}
We will prove Proposition \ref{prop:k2} 
by refering to the proof of \cite[Lemma 7.4]{decommer}, 
which was discussed with the help of graphical calculation. 
From the viewpoint of category theory, 
we proceed with a proof of Proposition \ref{prop:k2}.
\end{remark}
\begin{proof}
In view of \eqref{eq:monoid2},
the right-hand-side of \eqref{eq:mXk2} is
\begin{align}
&((m(\eta\otimes 1_X)l_X^{-1})\otimes 1_X)(1_A\otimes m_X)
a_{AAX}(\sigma\otimes 1_X)a_{AAX}^{-1}(1_A\otimes k)a_{AAX}
\label{eq:morphism1}
\\
&\quad
(\sigma\otimes 1_X)a_{AAX}^{-1}
\notag
\\
=&
(m\otimes m_X)a_{A\otimes AAX}((a_{AAA}^{-1}(1_A\otimes\sigma)
a_{AAA})\otimes 1_X)a_{A\otimes AAX}^{-1}
(1_{A\otimes A}\otimes k)
\notag
\\
&\quad
a_{A\otimes AAX}((a_{AAA}^{-1}(1_A\otimes\sigma)
a_{AAA}(((\eta\otimes 1_A)l_A^{-1})\otimes 1_A))\otimes 1_X)a_{AAX}^{-1}.
\notag
\end{align}
In fact, we can prove it by
using $(((\eta\otimes 1_A)l_A^{-1})\otimes 1_A)\sigma
=
a_{AAA}^{-1}(1_A\otimes\sigma)a_{AAA}(((\eta\otimes 1_A)l_A^{-1})
\otimes 1_A)$,
which comes from the fact that
$l_{A\otimes A}=(l_A\otimes 1_A)a_{IAA}^{-1}$
\cite[Lemma XI.2.2]{kassel}.

Because $\sigma$ is an isomorphism and satisfies \eqref{eq:braidedmonoid4},
the right-hand-side of \eqref{eq:morphism1}
is
\begin{align*}
&(m\otimes m_X)a_{A\otimes AAX}((a_{AAA}^{-1}(1_A\otimes\sigma)
a_{AAA})\otimes 1_X)a_{A\otimes AAX}^{-1}
(1_{A\otimes A}\otimes k)
a_{A\otimes AAX}
\\
&\quad
((\sigma\otimes 1_A)\otimes 1_X)
(((\sigma^{-1}\otimes 1_A)a_{AAA}^{-1}(1_A\otimes\sigma)
a_{AAA}(\sigma\otimes 1_A))\otimes 1_X)
\\
&\quad
((((1_A\otimes\eta)r_A^{-1})\otimes 1_A)\otimes 1_X)
a_{AAX}^{-1},
\end{align*}
which is
\begin{align}
&(1_A\otimes m_X)a_{AAX}
(((m\otimes 1_A)a_{AAA}^{-1}(1_A\otimes\sigma)
a_{AAA}(\sigma\otimes 1_A))\otimes 1_X)
\label{eq:morphism2}
\\
&\quad
a_{A\otimes AAX}^{-1}(1_{A\otimes A}\otimes k)a_{A\otimes AAX}
\notag
\\
&\quad((a_{AAA}^{-1}(1_A\otimes\sigma)a_{AAA}(\sigma\otimes 1_A)
a_{AAA}^{-1}(1_A\otimes\sigma^{-1})a_{AAA})\otimes 1_X)
\notag
\\
&\quad
((((1_A\otimes\eta)r_A^{-1})\otimes 1_A)\otimes 1_X)
a_{AAX}^{-1}
\notag
\end{align}
owing to \eqref{eq:braidrel}.
On account of 
\eqref{eq:braidedmonoid2}
and \eqref{eq:mYsigma},
\eqref{eq:morphism2}
is
\begin{align}
&m_Y^\sigma a_{AAX}
(((1_A\otimes m)a_{AAA})\otimes 1_X)a_{A\otimes AAX}^{-1}
(1_{A\otimes A}\otimes k)
a_{A\otimes AAX}
\label{eq:morphism3}
\\
&\quad
((a_{AAA}^{-1}(1_A\otimes\sigma)a_{AAA}(\sigma\otimes 1_A)
a_{AAA}^{-1}(1_A\otimes\sigma^{-1})a_{AAA})\otimes 1_X)
\notag
\\
&\quad
((((1_A\otimes\eta)r_A^{-1})\otimes 1_A)\otimes 1_X)
a_{AAX}^{-1}
\notag
\\
=&m_Y^\sigma
(1_A\otimes ((m\otimes 1_X)a_{AAX}^{-1}(1_A\otimes k)
a_{AAX}(\sigma\otimes 1_X)))
a_{AA\otimes AX}
\notag
\\
&\quad
((a_{AAA}(\sigma\otimes 1_A)a_{AAA}^{-1}(1_A\otimes\sigma^{-1})a_{AAA})
\otimes 1_X)
\notag
\\
&\quad
((((1_A\otimes\eta)r_A^{-1})\otimes 1_A)\otimes 1_X)
a_{AAX}^{-1}.
\notag
\end{align}

We note that 
\eqref{eq:mYk1} holds because of 
\eqref{eq:braidcommmYmYtriv},
\eqref{eq:k},
and
\eqref{eq:monoid2}.
It follows from \eqref{eq:mYk1} that
the right-hand-side of \eqref{eq:morphism3}
is 
$m_Y^\sigma
(1_A\otimes m_Y)
a_{AAY}(\sigma\otimes 1_Y)
a_{A\otimes AAX}
((a_{AAA}^{-1}(1_A\otimes\sigma^{-1})a_{AAA})
\otimes 1_X)
((((1_A\otimes\eta)r_A^{-1})\otimes 1_A)\otimes 1_X)
a_{AAX}^{-1}$,
which is exactly the same as
\begin{align}
&(m\otimes 1_X)a_{AAX}^{-1}
(1_A\otimes k)
a_{AAX}
(\sigma\otimes 1_X)
a_{AAX}^{-1}
\label{eq:morphism4}
\\
&\quad
(1_A\otimes((1_A\otimes m_X)a_{AAX}(\sigma\otimes 1_X)))
a_{AA\otimes AX}
\notag
\\
&\quad
(((1_A\otimes\sigma^{-1})a_{AAA})\otimes 1_X)
((((1_A\otimes\eta)r_A^{-1})\otimes 1_A)\otimes 1_X)
a_{AAX}^{-1}
\notag
\\
=&
(m\otimes 1_X)a_{AAX}^{-1}
(1_A\otimes k)
a_{AAX}
(\sigma\otimes 1_X)
a_{AAX}^{-1}
\notag
\\
&\quad
(1_A\otimes((1_A\otimes m_X)a_{AAX}))
a_{AA\otimes AX}
(a_{AAA}\otimes 1_X)
\notag
\\
&\quad((((1_A\otimes\eta)r_A^{-1})\otimes 1_A)\otimes 1_X)
a_{AAX}^{-1}
\notag
\\
=&
(m\otimes 1_X)a_{AAX}^{-1}
(1_A\otimes k)
a_{AAX}
((\sigma(1_A\otimes\eta))\otimes 1_X)
a_{AIX}^{-1}
\notag
\\
&\quad
(1_A\otimes((1_I\otimes m_X)a_{IAX}))a_{AI\otimes AX}
(a_{AIA}\otimes 1_X)
((r_A^{-1}\otimes 1_A)\otimes 1_X)
\notag
\\&\quad
a_{AAX}^{-1}
\notag
\end{align}
in view of
\eqref{eq:k},
\eqref{eq:braidcommmYmYsigma},
and
\eqref{eq:mYsigma}.

Because of 
\eqref{eq:braidedmonoid4},
the right-hand-side of \eqref{eq:morphism4}
is the morphism
\begin{align*}
&((m(\eta\otimes 1_A))\otimes 1_X)a_{IAX}^{-1}
(1_I\otimes k)
a_{IAX}
((l_A^{-1}r_A)\otimes 1_X)
a_{AIX}^{-1}
\\
&\quad(1_A\otimes ((1_I\otimes m_X)a_{IAX}))
a_{AI\otimes AX}
(a_{AIA}\otimes 1_X)
((r_A^{-1}\otimes 1_A)\otimes 1_X)a_{AAX}^{-1}.
\end{align*}
By using
\eqref{eq:monoid2},
together with
the triangle axiom \eqref{eq:triangle}
and
$l_{A\otimes X}=(l_A\otimes 1_X)a_{IAX}^{-1}$
\cite[Lemma XI.2.2]{kassel},
we deduce
that
the above morphism is
the left-hand-side of \eqref{eq:mXk2}.
This proves the proposition.
\end{proof}
The equation \eqref{eq:braidrel} is exactly the braid relation in $\mathcal{C}$.
\begin{proposition}\label{prop:pre-reflection}
For $A, X\in\mathcal{C}$, we suppose that
morhpisms $m: A\otimes A\to A, m_X: A\otimes X\to X,
\sigma: A\otimes A\to A\otimes A$, and $k: A\otimes X\to A\otimes X$
of $\mathcal{C}$ satisfy
$\eqref{eq:mXk3}, \eqref{eq:mXk1}, \eqref{eq:mXk2}$,
and
\begin{equation}
m\sigma=m.
\label{eq:mAr}
\end{equation}
Then 
\begin{align}
&(m\otimes 1_X)a_{AAX}^{-1}(1_A\otimes k)a_{AAX}
(\sigma\otimes 1_X)a_{AAX}^{-1}
(1_A\otimes k)a_{AAX}
\label{eq:prereflection1}
\\
&\quad
(\sigma\otimes 1_X)a_{AAX}^{-1}
\notag
\\
=&
(m\otimes 1_X)a_{AAX}^{-1}(1_A\otimes k)a_{AAX}
(\sigma\otimes 1_X)a_{AAX}^{-1}
(1_A\otimes k),
\notag
\\
&(1_A\otimes m_X)a_{AAX}(\sigma\otimes 1_X)a_{AAX}^{-1}
(1_A\otimes k)a_{AAX}(\sigma\otimes 1_X)
\label{eq:prereflection2}
\\
&\quad
a_{AAX}^{-1}(1_A\otimes k)
\notag
\\
=&
(1_A\otimes m_X)a_{AAX}(\sigma\otimes 1_X)a_{AAX}^{-1}
(1_A\otimes k)a_{AAX}(\sigma\otimes 1_X)a_{AAX}^{-1}.
\notag
\end{align}
\end{proposition}
\begin{proof}
It follows from \eqref{eq:mXk1} that 
the left-hand-side of \eqref{eq:prereflection1} is
$k(m\otimes 1_X)(\sigma\otimes 1_X)a_{AAX}^{-1}$,
which is $k(m\otimes 1_X)a_{AAX}^{-1}$
because of \eqref{eq:mAr}.
By using \eqref{eq:mXk1} again,
we conclude that $k(m\otimes 1_X)a_{AAX}^{-1}$ is 
the right-hand-side of \eqref{eq:prereflection1}.

By making use of \eqref{eq:mXk3} and \eqref{eq:mXk2}
instead of \eqref{eq:mXk1} and \eqref{eq:mAr},
we can show \eqref{eq:prereflection2} in much the same way.
\end{proof}

A pair of morphisms $(f: X\to Y, g: X\to Y)$ of $\mathcal{C}$ 
is called monic,
iff this pair satisfies the following condition:
\[
fh=fh'\text{\ and\ }gh=gh'
\ (h, h': A\to X)\text{\ induce\ } h=h'. 
\]
This notion is a generalization of the usual monomorphism.
\begin{corollary}\label{thm:reflection}
  For $A, X\in\mathcal{C}$, we suppose that
  morhpisms $m: A\otimes A\to A, m_X: A\otimes X\to X,
  \sigma: A\otimes A\to A\otimes A$, and $k: A\otimes X\to A\otimes X$
  of $\mathcal{C}$ satisfy
  $\eqref{eq:mXk3}, \eqref{eq:mXk1}, \eqref{eq:mXk2}$,
  and
  $\eqref{eq:mAr}$.
  If the pair $(m\otimes 1_X, (1_A\otimes m_X)a_{AAX})$ of morphisms
is monic, then 
$k: A\otimes X\to A\otimes X$
is a solution to the reflection equation
associated with $\sigma$\/$:$
\begin{align}
&a_{AAX}^{-1}(1_A\otimes k)a_{AAX}(\sigma\otimes 1_X)a_{AAX}^{-1}
(1_A\otimes k)a_{AAX}(\sigma\otimes 1_X)a_{AAX}^{-1}
\label{eq:RE}
\\
=&
(\sigma\otimes 1_X)a_{AAX}^{-1}(1_A\otimes k)a_{AAX}
(\sigma\otimes 1_X)a_{AAX}^{-1}
(1_A\otimes k).
\notag
\end{align}
\end{corollary}
\begin{proof}
From Proposition \ref{prop:pre-reflection},
\eqref{eq:mXk3}, \eqref{eq:mAr},
\eqref{eq:prereflection1},
and \eqref{eq:prereflection2},
\begin{align*}
  &(m\otimes 1_X)a_{AAX}^{-1}(1_A\otimes k)a_{AAX}
  (\sigma\otimes 1_X)a_{AAX}^{-1}
  (1_A\otimes k)a_{AAX}
  \\
  &\quad
  (\sigma\otimes 1_X)a_{AAX}^{-1}
  \\
  =&
  (m\otimes 1_X)(\sigma\otimes 1_X)
  a_{AAX}^{-1}(1_A\otimes k)a_{AAX}
  (\sigma\otimes 1_X)a_{AAX}^{-1}
  (1_A\otimes k),
  \\
  &(1_A\otimes m_X)a_{AAX}(\sigma\otimes 1_X)a_{AAX}^{-1}
  (1_A\otimes k)a_{AAX}(\sigma\otimes 1_X)
  \\
  &\quad
  a_{AAX}^{-1}(1_A\otimes k)
  \\
  =&
  (1_A\otimes m_X)(1_A\otimes k)
  a_{AAX}(\sigma\otimes 1_X)a_{AAX}^{-1}
  (1_A\otimes k)a_{AAX}(\sigma\otimes 1_X)a_{AAX}^{-1}.
  \end{align*}
Because the pair $(m\otimes 1_X, (1_A\otimes m_X)a_{AAX})$
is monic, \eqref{eq:RE} follows.
\end{proof}

The dynamical reflection map is a solution to the reflection equation 
\eqref{eq:RE}
in the tensor category $\setH$.

We now present a sufficient condition for the 
pair $(m\otimes 1_X, (1_A\otimes m_X)a_{AAX})$
of the tensor category $\setH$ to be monic.
\begin{proposition}\label{cor:reflection}
For $A, X\in\setH$,
let
$m: A\otimes A\to A$
and $m_X: A\otimes X\to X$
be morphisms of $\setH$.
If the maps $A\ni b\mapsto m(\lambda)(a, b)\in A$ are injective for any
$\lambda\in H$ and $a\in A$, then
the pair $(m\otimes 1_X, (1_A\otimes m_X)a_{AAX})$ is monic.
\end{proposition}
\begin{proof}
Let $\lambda\in H$, $a_1, a_2, b_1, b_2\in A$, and $x_1, x_2\in X$.
It suffices to prove that
$(m\otimes 1_X)(\lambda)(a_1, b_1, x_1)=(m\otimes 1_X)(\lambda)(a_2, b_2, x_2)$
and 
$((1_A\otimes m_X)a_{AAX})(\lambda)(a_1, b_1, x_1)=((1_A\otimes m_X)a_{AAX})(\lambda)(a_2, b_2, x_2)$
induce $(a_1, b_1, x_1)=(a_2, b_2, x_2)$.

From these equations
\begin{align*}
&(m(\lambda)(a_1, b_1), x_1)=(m(\lambda)(a_2, b_2), x_2),
\\
&(a_1, m_X(\lambda a_1)(b_1, x_1))
=(a_2, m_X(\lambda a_2)(b_2, x_2)).
\end{align*}
Thus $a_1=a_2, x_1=x_2$,
and $m(\lambda)(a_1, b_1)=m(\lambda)(a_2, b_2)
=m(\lambda)(a_1, b_2)$.
Because the map $A\ni b\mapsto m(\lambda)(a, b)\in A$ is injective,
$b_1=b_2$.
This completes the proof.
\end{proof}

The above proposition is useful for the construction
of dynamical Yang-Baxter maps
in Section \ref{section:construction}
(See Proposition \ref{prop:braidedmonoid}).
\section{Braided monoids}
\label{section:braided}
In this section,
we focus on the relations \eqref{eq:monoid2-2}
and \eqref{eq:monoid2}--\eqref{eq:braidedmonoid4},
which can produce monoids in the tensor category
(Cf.\ \cite[Section 5]{decommer}).

Let $\mathcal{C}$ be a tensor category.
\begin{definition}\label{def:monoid}
An object $A$ of $\mathcal{C}$, together with 
morphisms $m: A\otimes A\to A$
and
$\eta: I\to A$, is a monoid,
iff the morphisms satisfy
\eqref{eq:monoid2}
and
\begin{equation}
m(m\otimes 1_A)=m(1_A\otimes m)a_{AAA}.
\label{eq:asslaw} 
\end{equation}
\end{definition}

This monoid is also called a ring in \cite[Definition 4.3.1]{kashiwara}.
\begin{definition}\label{def:braided}
A monoid $(A, m, \eta)$ with a morphism
$\sigma: A\otimes A\to A\otimes A$ is braided,
iff the morphisms $m$, $\eta$, and $\sigma$ satisfy
\eqref{eq:braidedmonoid2},
\eqref{eq:braidedmonoid4},
and
\begin{align}
&(1_A\otimes m)a_{AAA}(\sigma\otimes 1_A)a_{AAA}^{-1}(1_A\otimes\sigma)
=\sigma(m\otimes 1_A)a_{AAA}^{-1};
\label{eq:braidedmonoid1} 
\\
&\sigma(\eta\otimes 1_A)=(1_A\otimes\eta)r_A^{-1}l_A.
\label{eq:braidedmonoid3} 
\end{align}
\end{definition}

For braided semigroups in the tensor category,
see \cite{matsumoto}.
\begin{proposition}\label{prop:twistedmonnoid}
If $(A, m, \eta, \sigma)$ is a braided monoid in $\mathcal{C}$,
then
$A\otimes A$
is a monoid in $\mathcal{C}$
with morphisms\/$:$
\begin{align}
m_{A\otimes A}
=&(m\otimes m)a_{A\otimes A AA}(a_{AAA}^{-1}\otimes 1_A)
((1_A\otimes\sigma)\otimes 1_A) 
\label{eq:mAotimesA1}
\\
&\quad 
(a_{AAA}\otimes 1_A)a_{A\otimes A AA}^{-1};
\notag
\\
\eta_{A\otimes A}=&(\eta\otimes\eta)l_I^{-1}(=(\eta\otimes\eta)r_I^{-1}).
\label{eq:etaAotimesA1}
\end{align}
\end{proposition}
\begin{proof}
For the proof, it suffices to show:
\begin{align}
&m_{A\otimes A}(m_{A\otimes A} \otimes 1_{A \otimes A}) = m_{A\otimes A}
(1_{A \otimes A} \otimes m_{A\otimes A})a_{A \otimes A A \otimes A A \otimes A};
\label{eq:mAotimesA}
\\
&
m_{A \otimes A}(\eta_{A \otimes A} \otimes 1_{A \otimes A})l_{A \otimes A}^{-1}
=1_{A \otimes A}=
m_{A \otimes A}(1_{A \otimes A}\otimes\eta_{A\otimes A})r_{A \otimes A}^{-1}.
\label{eq:etaAotimesA}
\end{align}
On account of
\eqref{eq:braidedmonoid4} and \eqref{eq:braidedmonoid3},
it is a simple matter to show \eqref{eq:etaAotimesA}.
We prove \eqref{eq:mAotimesA}
only, when the tensor category $\mathcal{C}$ is strict
\cite[Definition XI.2.1]{kassel};
hence, we assume in this proof that the associativity constraint $a$
and the unit constraints $l, r$ are all identities.

From \eqref{eq:mAotimesA1}, the left-hand-side of
$\eqref{eq:mAotimesA}$ is
\begin{align*}
&(m \otimes m)
(1_A\otimes\sigma\otimes 1_A)
(((m \otimes m)
(1_A\otimes\sigma\otimes 1_A))\otimes 1_{A\otimes A})
\\
=& (m \otimes m) 
(1_A\otimes\sigma(m\otimes 1_A)\otimes 1_A)
(m\otimes 1_{A\otimes A\otimes A\otimes A})(1_A\otimes\sigma\otimes
1_{A\otimes A\otimes A}).
\end{align*}
By virtue of \eqref{eq:braidedmonoid1},
the right-hand-side of the above equation is
\begin{align*}
&(m \otimes m) 
(1_A\otimes((1_A\otimes m)(\sigma\otimes 1_A)(1_A\otimes\sigma))
\otimes 1_A)
(m\otimes 1_{A\otimes A\otimes A\otimes A})
\\
&\quad(1_A\otimes\sigma\otimes
1_{A\otimes A\otimes A})
\\
=&
(m(m\otimes 1_A) \otimes 1_A) 
(1_{A\otimes A\otimes A}\otimes m(m\otimes 1_A))
\\
&\quad(1_{A\otimes A}\otimes((\sigma\otimes 1_A)(1_A\otimes\sigma))
\otimes 1_A)(1_A\otimes\sigma\otimes
1_{A\otimes A\otimes A}),
\end{align*}
which coincides with
\begin{align}
&(m(1_A\otimes m) \otimes 1_A) 
(1_{A\otimes A\otimes A}\otimes m(1_A\otimes m))
\label{eq:pfmAotimesA}
\\
&\quad
(1_{A\otimes A}\otimes((\sigma\otimes 1_A)(1_A\otimes\sigma))
\otimes 1_A)(1_A\otimes\sigma\otimes
1_{A\otimes A\otimes A})
\notag
\\
=&
(m \otimes 1_A) 
(1_{A\otimes A}\otimes m(1_A\otimes m))
\notag
\\
&\quad
(1_A\otimes ((m \otimes 1_A) 
(1_{A}\otimes\sigma)(\sigma\otimes
1_A))\otimes 1_{A\otimes A})(1_{A\otimes A\otimes A}\otimes\sigma\otimes 1_A)
\notag
\end{align}
in view of \eqref{eq:asslaw}.
Because of $\eqref{eq:braidedmonoid2}$,
the right-hand-side of  $\eqref{eq:pfmAotimesA}$ is
\begin{align*}
&(m \otimes 1_A) 
(1_{A\otimes A}\otimes m(1_A\otimes m))
(1_A\otimes (\sigma(1_A\otimes m) )\otimes 1_{A\otimes A})
(1_{A\otimes A\otimes A}\otimes\sigma\otimes 1_A)
\\
&=
(m \otimes m) 
(1_A\otimes\sigma\otimes 1_A)
(1_{A\otimes A}\otimes ((m\otimes m)(1_A\otimes\sigma\otimes 1_A))),
\end{align*}
which is exactly the right-hand-side of \eqref{eq:mAotimesA}
with the aid of \eqref{eq:mAotimesA1}.
This completes the proof.
\end{proof}
The monoid in the above proposition is called a twisted monoid
and denoted by $A\underset{\mathrm{tw}}{\otimes}A$.
\section{Left modules of monoids}
\label{section:leftmod}
In this section, we introduce the notion of left modules of monoids
in an arbitrary tensor category $\mathcal{C}$
(Cf.\ \cite[Sections 7 and 8]{decommer}).

Let $(A, m, \eta)$ be a monoid in $\mathcal{C}$
(Definition \ref{def:monoid}).
\begin{definition}
An object $X$ of $\mathcal{C}$ with 
a morphism $m_X: A\otimes X\to X$
is a left $A$-module in $\mathcal{C}$,
iff the morphisms satisfy
\eqref{eq:leftmod2}
and
\begin{equation}
m_X(m\otimes 1_X)=m_X(1_A\otimes m_X)a_{AAX}.
\label{eq:leftmod1}
\end{equation}
\end{definition}

We note that \eqref{eq:leftmodmY} holds, if $(Y, m_Y)$
is a left $A$-module in $\mathcal{C}$.
\begin{proposition}\label{prop:mYtriv}
For a monoid $(A, m, \eta)$
and its left module $(X, m_X)$ in $\mathcal{C}$,
we set $Y=A\otimes X$.
Then $(Y, m_Y^{\mathrm{triv}})$ is a left $A$-module.
Here, $m_Y^{\mathrm{triv}}: A\otimes Y\to Y\in\setH$ is
defined by $\eqref{eq:mYtriv}$.
\end{proposition}
\begin{proof}
Combining $\eqref{eq:mYtriv}$ and $\eqref{eq:asslaw}$
with
the pentagon axiom \eqref{eq:pentagon}
yields
$m_{Y}^{\mathrm{triv}}
(m \otimes 1_{Y}) 
= 
m_{Y}^{\mathrm{triv}}(1_{A} \otimes m_{Y}^{\mathrm{triv}})a_{A A Y}$.
By using $\eqref{eq:mYtriv}$ and $\eqref{eq:monoid2}$, together with 
$l_{Y}=(l_A\otimes 1_X)a_{IAX}^{-1}$
\cite[Lemma XI.2.2]{kassel}, we see at once that 
$m_{Y}^{\mathrm{triv}}(\eta \otimes 1_{Y}) = l_{Y}$.
\end{proof}
\begin{proposition}
For a braided monoid $(A, m, \eta, \sigma)$
$($$\mathrm{Definition}$ $\ref{def:braided}$$)$
and its left module $(X, m_X)$ in $\mathcal{C}$,
we set $Y=A\otimes X$.
Then $(Y, m_Y^{\sigma})$ is a left $A$-module.
Here,
$m_Y^{\sigma}: A\otimes Y\to Y\in\setH$ is defined by $\eqref{eq:mYsigma}$.
\end{proposition}
\begin{proof}
We prove the relations below:
\begin{align}
&m_{Y}^\sigma(m \otimes 1_{Y}) = m_{Y}^\sigma(1_{A} \otimes m_{Y}^\sigma)a_{A A Y};
\label{eq:mYsigmaaction1}
\\
&m_{Y}^\sigma(\eta \otimes 1_{Y}) = l_{Y}.
\label{eq:mYsigmaaction2}
\end{align}

We only prove \eqref{eq:mYsigmaaction1}.
On account of \eqref{eq:mYsigma} and \eqref{eq:leftmod1},
the right-hand-side of \eqref{eq:mYsigmaaction1} is
\begin{align}
& (1_{A} \otimes (m_{X}(1_{A} \otimes m_{X})))
a_{A A Y}(\sigma \otimes 1_{Y})a_{A A Y}^{-1}
\label{eq:pfmYsigma1}
\\
&\quad
(1_{A} \otimes (a_{A A X}(\sigma \otimes 1_{X})a_{A A X}^{-1}))
a_{A A Y} 
\notag
\\
=& (1_{A} \otimes (m_{X}(m \otimes 1_{X})a_{A A X}^{-1}))
a_{A A Y}(\sigma \otimes 1_{Y})a_{A A Y}^{-1}  
\notag 
\\
&\quad
 (1_{A} \otimes (a_{A A X} (\sigma \otimes 1_{X})a_{A A X}^{-1}))
 a_{A A Y}
\notag
\\
=&
(1_A\otimes m_X)a_{AAX}(((1_A\otimes m)a_{AAA}(\sigma\otimes 1_A)
a_{AAA}^{-1}(1_A\otimes\sigma))\otimes 1_X)
\notag
\\
&\quad a_{AA\otimes AX}^{-1}(1_A\otimes a_{AAX}^{-1})a_{AAY}.
\notag
\end{align}
By using \eqref{eq:braidedmonoid1}, we see that the right-hand-side of \eqref{eq:pfmYsigma1} is
\[
(1_{A} \otimes m_{X})a_{A A X}((\sigma (m\otimes 1_A)a_{AAA}^{-1})
\otimes 1_{X})
a_{A A \otimes A X}^{-1}(1_A\otimes a_{AAX}^{-1})a_{AAY} ,
\]
which is exactly
the left-hand-side of \eqref{eq:mYsigmaaction1}
because of the pentagon axiom
\eqref{eq:pentagon} and $Y=A\otimes X$.
%
\end{proof}
\begin{definition}\label{def:braid-commute}
Let $(V, m_V)$ and $(V, m'_V)$ be left modules
of a monoid $(A, m, \eta)$ in $\mathcal{C}$,
and let $\sigma: A\otimes A\to A\otimes A$ be a morphism of $\mathcal{C}$.
A pair $(m_V, m'_V)$
braid-commutes,
iff
$m_V(1_A\otimes m'_V)=m'_V(1_A\otimes m_V)a_{AAV}(\sigma\otimes 1_V)
a_{AAV}^{-1}$.
\end{definition}
\begin{remark}\label{rem:braidcommmYmYtriv}
Let $(X, m_X)$ and $(Y, m_Y)$ be left modules
of a monoid $(A, m, \eta)$ in $\mathcal{C}$,
and let $\sigma: A\otimes A\to A\otimes A$ be a morphism of $\mathcal{C}$.
We note that 
\eqref{eq:braidcommmYmYtriv} holds, if and only if
the pair
$(m_Y, m_Y^{\mathrm{triv}})$ braid-commutes.
\end{remark}
\begin{remark}\label{rem:braidcommmYmYsigma}
Let $(X, m_X)$ and $(Y, m_Y)$ be left modules
of a braided monoid $(A, m, \eta, \sigma)$ in $\mathcal{C}$.
Then \eqref{eq:braidcommmYmYsigma} is a necessary and sufficient condition
for
$(m_Y, m_Y^{\sigma})$ to braid-commute.
\end{remark}

Let $(A, m, \eta, \sigma)$ be a braided monoid in $\mathcal{C}$
and let $(X, m_X)$ be a left $(A, m, \eta)$-module in $\mathcal{C}$.
Write $Y=A\otimes X$.

We will use Theorems \ref{thm:mYthetaY},
\ref{thm:thetaYmY}
and Corollary \ref{cor:mYthetaY}
for the proof of Corollary \ref{cor:equivmYshikaku}.
\begin{theorem}\label{thm:mYthetaY}
Let $(Y, m_Y)$ be a left $(A, m, \eta)$-module in $\mathcal{C}$
satisfying $\eqref{eq:mXmY}$
and
that the pair
$(m_Y, m_Y^{\mathrm{triv}})$ braid-commutes
$($For $m_Y^{\mathrm{triv}}$, see $\eqref{eq:mYtriv}$$)$.
We set 
\begin{equation}\label{eq:theta}
\theta_Y=m_Y^{\mathrm{triv}}(1_A\otimes m_Y)a_{AAY}: 
(A\otimes A)\otimes Y\to Y.
\end{equation}
Then $(Y, \theta_Y)$ is a left module of the twisted monoid
$A\underset{\mathrm{tw}}{\otimes}A$ $($See the proof of $\mathrm{Proposition}$
$\ref{prop:twistedmonnoid}$
below$)$
in 
$\mathcal{C}$ satisfying\/$:$
\begin{align}\label{eq:thetamodule1}
&m_X\theta_Y=m_X(1_A\otimes m_X)a_{AAX}(1_{A\otimes A}\otimes m_X);
\\\label{eq:thetamodule2}
&\theta_Y((1_A\otimes\eta)\otimes 1_Y)=(m\otimes 1_X)a_{AAX}^{-1}
(r_A\otimes 1_Y).
\end{align}
\end{theorem}
\begin{proof}
For simplicity we assume that the tensor category $\mathcal{C}$ is strict \cite[Definition XI.2.1]{kassel}.
In order to prove that $(Y, \theta_Y)$ is a left module of
$A\underset{\mathrm{tw}}{\otimes}A$, we show:
\begin{align}
&\theta_{Y}(m_{A \otimes A} \otimes 1_{Y})=  \theta_{Y}(1_{A \otimes A} \otimes \theta_{Y})
; 
\label{eq:thetaYaction1} 
\\
&\theta_{Y}(\eta_{A \otimes A} \otimes 1_{Y}) = 
1_{Y}.
 \label{eq:thetaYaction2}
\end{align}
For the morphisms $m_{A\otimes A}$ and $\eta_{A\otimes A}$,
see \eqref{eq:mAotimesA1}
and \eqref{eq:etaAotimesA1}.

We can easily show \eqref{eq:thetaYaction2}
by means of
\eqref{eq:theta},
Proposition \ref{prop:mYtriv}, 
and the fact that $(Y, m_Y)$
is a left $(A, m, \eta)$-module.

We next prove \eqref{eq:thetaYaction1}.
On account of
Remark \ref{rem:braidcommmYmYtriv}
and 
\eqref{eq:theta}, 
the right-hand-side of \eqref{eq:thetaYaction1} is  
\begin{align}
&m_Y^{\mathrm{triv}} (1_{A} \otimes 
(m_{Y}(1_A\otimes m_Y^{\mathrm{triv}})))
(1_{A \otimes A} \otimes 1_{A} \otimes m_{Y})
\label{eq:pfthetaYaction1}
\\
=& m_Y^{\mathrm{triv}}(1_{A} \otimes (m_Y^{\mathrm{triv}}
(1_{A} \otimes m_{Y})
(\sigma \otimes 1_{Y})
)) 
\notag 
(1_{A \otimes A} \otimes 1_{A} \otimes m_{Y}
)
\\
=& m_Y^{\mathrm{triv}}(1_A\otimes m_Y^{\mathrm{triv}}) 
(1_A\otimes 1_A\otimes(m_Y(1_A\otimes m_Y)))
(1_A\otimes \sigma \otimes 1_{A\otimes Y}).
\notag
\end{align}
Because $(Y, m_Y)$ and
$(Y, m_Y^{\mathrm{triv}})$
are left $(A, m, \eta)$-modules
(See Proposition \ref{prop:mYtriv}), 
the right-hand-side of \eqref{eq:pfthetaYaction1} is
\begin{align*}
&m_Y^{\mathrm{triv}}(m \otimes 1_{Y}) 
(1_{A} \otimes 1_{A} \otimes (m_{Y}(m\otimes 1_Y)
))
(1_A\otimes \sigma \otimes 1_{A \otimes Y})
\\
=&m_Y^{\mathrm{triv}}(1_A\otimes m_Y)(m \otimes 1_{A\otimes Y})
(1_{A} \otimes 1_{A} \otimes m \otimes 1_{Y}
)
 (1_A\otimes \sigma \otimes 1_{A \otimes Y}),
\end{align*}
which is exactly
\begin{align*}
&\theta_Y
(m \otimes 1_{A\otimes Y})
(1_{A} \otimes 1_{A} \otimes m \otimes 1_{Y}
)
 (1_A\otimes \sigma \otimes 1_{A \otimes Y})
\\
=&
\theta_Y
(m\otimes m \otimes 1_Y)
 (1_A\otimes \sigma \otimes 1_{A \otimes Y})
\end{align*}
in view of \eqref{eq:theta}.
\eqref{eq:thetaYaction1} thus holds.

From \eqref{eq:mXmY}, \eqref{eq:mYtriv},  \eqref{eq:leftmod1}, and \eqref{eq:theta}, 
we can show \eqref{eq:thetamodule1}.
%
By means of \eqref{eq:mYtriv},
 \eqref{eq:theta},
 and the fact that $(Y, m_Y)$ is a left $(A, m, \eta)$-module,
we see that \eqref{eq:thetamodule2} holds.
The proof is therefore complete.
\end{proof}
\begin{theorem}\label{thm:thetaYmY}
Let $(Y, \theta_Y)$ be a left $A\underset{\mathrm{tw}}{\otimes}A$-module
satisfying $\eqref{eq:thetamodule1}$
and $\eqref{eq:thetamodule2}$.
Write
\begin{equation}\label{eq:mYtensor}
m_Y=\theta_Y(((\eta\otimes 1_A)l_A^{-1})\otimes 1_Y): A\otimes Y\to Y.
\end{equation}
Then $(Y, m_Y)$ is a left $A$-module in 
$\mathcal{C}$ satisfying
$\eqref{eq:mXmY}$
and that
the pair $(m_Y, m_Y^{\mathrm{triv}})$ braid-commutes.
\end{theorem}
\begin{proof}
We assume that the tensor category $\mathcal{C}$ is strict.
On account of \eqref{eq:leftmod2}
and
\eqref{eq:thetamodule1}, 
it is a simple matter to show \eqref{eq:mXmY}.

We next prove that $(Y, m_Y)$ is a left $A$-module.
For this purpose, we show the relation below only.
\begin{equation}
m_{Y}(m \otimes 1_{Y}) = m_{Y}(1_{A} \otimes m_{Y})
.\label{eq:pfmYaction1}
\end{equation}

On account of \eqref{eq:mAotimesA1} and the fact that $(Y, \theta_Y)$ is a left $A\underset{\mathrm{tw}}{\otimes}A$-module,
the right-hand-side of \eqref{eq:pfmYaction1} is
\begin{align}\label{eq:pfmYisaction1}
& \theta_{Y} (1_{A \otimes A} \otimes \theta_{Y})  
((\eta \otimes 1_{A})
\otimes 1_A\otimes 1_{A \otimes Y}) 
 ((1_{A} \otimes \eta) \otimes 1_{A \otimes Y})
  \\
=& \theta_{Y} (m_{A \otimes A} \otimes 1_{Y}) 
((\eta \otimes 1_{A})
\otimes 1_A\otimes 1_{A \otimes Y}) 
  ((1_{A} \otimes \eta) \otimes 1_{A \otimes Y})
  \notag
  \\
=& \theta_{Y}(1_{A} \otimes m \otimes 1_{Y})
(m\otimes 1_A
\otimes 1_A\otimes 1_Y)
\notag
\\
&\quad
(((1_A\otimes\sigma)
(1_{A\otimes A}\otimes\eta))
\otimes 1_{A \otimes Y})
((\eta\otimes 1_A)
\otimes 1_{A \otimes Y}).
\notag
\end{align}

From
\eqref{eq:monoid2}
and
\eqref{eq:braidedmonoid4},
\begin{equation}
(m\otimes 1_A)
(1_A\otimes \sigma)
(1_{A\otimes A}\otimes\eta)=
1_{A \otimes A}.
\label{eq:msigmaetar}
\end{equation}
The right-hand-side of $\eqref{eq:pfmYisaction1}$ is hence
$\theta_{Y} (1_{A} \otimes m \otimes 1_{Y}) 
(\eta\otimes 1_A
\otimes 
1_{A} \otimes 1_{Y}),$
%
which 
is exactly the same as the left-hand-side of \eqref{eq:pfmYaction1}.

The task is now to prove that the pair $(m_Y, m_Y^{\mathrm{triv}})$ 
braid-commutes. 
For the proof,
we show that either side of \eqref{eq:braidcommmYmYtriv}
is
$\theta_Y(\sigma\otimes 1_Y)
$.

It follows from $\eqref{eq:thetamodule2}$
that
\begin{equation}
m_{Y}^{\mathrm{triv}} = \theta_{Y} ((1_{A} \otimes \eta)
\otimes 1_{Y}),
\label{eq:mYtrivexp}
\end{equation} 
and
the right-hand-side of \eqref{eq:braidcommmYmYtriv} is
consequently
\begin{align}
\label{eq:pfmYmYtrivbraidcomm1}
 &\theta_{Y} ((1_{A} \otimes \eta)
  \otimes 1_{Y})  
 (1_{A} \otimes (\theta_{Y}((\eta \otimes 1_{A})
  \otimes 1_{Y}))) 
%
(\sigma \otimes 1_{Y}) 
 \\
=& \theta_{Y} (1_{A \otimes A} \otimes  \theta_{Y})
((1_{A} \otimes \eta)
 \otimes 1_{A \otimes A \otimes Y})  
\notag
(1_{A} \otimes (\eta \otimes 1_{A})
 \otimes 1_{Y})
  (\sigma \otimes 1_{Y}) 
\\
=& \theta_{Y} (m_{A \otimes A} \otimes 1_{Y})
  ((1_{A} \otimes \eta)
  \otimes 1_{A \otimes A \otimes Y}) 
  (1_{A} \otimes (\eta \otimes 1_{A})
   \otimes 1_{Y})
  (\sigma \otimes 1_{Y}) 
\notag
\\
=&
\theta_Y
(m\otimes m\otimes 1_Y)
\notag
(((1_A\otimes\sigma)
(1_{A}\otimes (1_{A} \otimes \eta)))
\otimes 1_A\otimes 1_Y)
\notag
\\
&\quad
((1_A\otimes\eta)
\otimes
1_{A}
 \otimes 1_{Y})
(\sigma\otimes 1_Y)
  \notag
\end{align}
because of \eqref{eq:mAotimesA1}
and the fact that 
$\theta_{Y}$ is a left $A\underset{\mathrm{tw}}{\otimes}A$-module. 

From \eqref{eq:monoid2}
and
\eqref{eq:msigmaetar},
the right-hand-side of \eqref{eq:pfmYmYtrivbraidcomm1} is
\begin{align*}
&
\theta_Y(1_A\otimes m\otimes 1_Y)
((1_A\otimes\eta)
\otimes
1_{A}
\otimes 1_Y)
(\sigma\otimes 1_Y)
\\
=&
\theta_Y(1_A\otimes (m(\eta\otimes 1_A))\otimes 1_Y)
(\sigma\otimes 1_Y)
\\
=&
\theta_Y
(\sigma\otimes 1_Y).
\end{align*}

On account of
\eqref{eq:mYtrivexp}
and the fact that
$m_Y=\theta_Y((\eta\otimes 1_A)
\otimes 1_Y)$,
the left-hand-side of \eqref{eq:braidcommmYmYtriv} is
$\theta_{Y}(1_{A \otimes A} \otimes \theta_{Y})
((\eta \otimes 1_{A})
 \otimes 
 ((1_{A} \otimes \eta)
 \otimes 1_{Y}))$,
which is
\begin{align}
& \theta_{Y}(m_{A\otimes A} \otimes 1_{Y})
((\eta \otimes 1_{A})
\otimes 
((1_{A} \otimes \eta)
 \otimes 1_{Y}))
\label{eq:pfmYmYtrivbraidcomm3}
 \\
=&
\theta_{Y}(((m (\eta\otimes 1_A))\otimes (m(1_A\otimes\eta))) \otimes 1_Y)
(\sigma
\otimes 1_Y)
\notag
\\
=&
\theta_{Y}
(\sigma
\otimes 1_Y)
\notag
\end{align}
owing to
\eqref{eq:monoid2},
\eqref{eq:mAotimesA1},
and the fact that
$\theta_{Y}$ is a left $A\underset{\mathrm{tw}}{\otimes}A$-module.
This completes the proof.
\end{proof}
\begin{corollary}\label{cor:mYthetaY}
The correspondence in Theorem $\ref{thm:mYthetaY}$
is the inverse of that in Theorem $\ref{thm:thetaYmY}$
and vice versa.
\end{corollary}

The proof of this corollary is straightforward.
\section{Construction}
\label{section:construction}
This section deals with the construction of dynamical reflection maps
by means of results in the preceding sections.
In this paper, we focus on the following
solutions to the braid relation \eqref{eq:braidrel} in the tensor category
$\setH$,
called
dynamical Yang-Baxter maps
\cite{shibu1,shibu2}.
\begin{remark}
The dynamical Yang-Baxter map in this paper
is called a dynamical braiding map in \cite{shibu1,shibu2}.
\end{remark}

Let 
$L$ be a set with a binary operation
$\cdot: L\times L\ni (a, b)\mapsto ab\in L$
and an element $e_L(\in L)$.
\begin{definition}\label{def:leftquasig}
The triplet $(L, \cdot, e_L)$ is a left quasigroup with a unit,
iff
there uniquely exists the element $b\in L$ such that $ab=c$
for all $a, c\in L$
and
$ae_L=e_L a=a$
for every $a\in L$.
\end{definition}
For $a, c\in L$,
let $a\backslash c$ denote the unique element $b\in L$
satisfying $ab=c$.
Hence,
$a(a\backslash c)=c$
and
$a\backslash(ac)=c$.

Every group is a left quasigroup with a unit.
However, there exists a left quasigroup with a unit that is not associative.
\begin{example}\label{leftquasiexample}
Let 
$L = \{e_{L}, l_{1}, l_{2}, l_{3}, l_{4}, l_{5} \}$
denote the set of six elements with the binary operation
$\cdot: L \times L \to L$
defined by Table 1. Here, $l_{2} \cdot l_{3} = l_{1}$ and $l_{3} \cdot l_{2} = e_{L}$. 
This 
$(L, \cdot, e_{L})$ 
is a left quasigroup with a unit $e_{L}$. 
Because $(l_1 \cdot l_2) \cdot l_3 = l_2$ and  $\ l_1 \cdot (l_2 \cdot l_3) = l_5$, 
it is not associative. 
\begin{table}[t]
\begin{center} 
  \begin{tabular}{l | r  r  r  r  r  r } 
\hline
    & $e_{L}$ & $l_{1}$ & $l_{2}$ & $l_{3}$ & $l_{4}$ & $l_{5}$ \\  \hline
   $e_{L}$ & $e_{L}$ & $l_{1}$ & $l_{2}$ & $l_{3}$ & $l_{4}$   & $l_{5}$\\ 
   $l_{1}$ & $l_{1}$ & $l_{5}$ & $l_{3}$ & $l_{4}$ & $l_{2}$   & $e_{L}$\\  
   $l_{2}$ & $l_{2}$ & $l_{3}$ & $l_{5}$ & $l_{1}$ & $e_{L}$   & $l_{4}$ \\ 
   $l_{3}$ & $l_{3}$ & $l_{4}$ & $e_{L}$ & $l_{2}$ & $l_{5}$ & $l_{1}$ \\ 
   $l_{4}$ & $l_{4}$ & $e_{L}$ & $l_{1}$ & $l_{5}$ & $l_{3}$ & $l_{2}$ \\ 
   $l_{5}$ & $l_{5}$ & $l_{2}$ & $l_{4}$ & $e_{L}$ & $l_{1}$ & $l_{3}$ \\ \hline
  \end{tabular}
\caption{The binary operation on $L$}
\end{center}
\end{table}
\end{example}
We write $H=L$.
\begin{proposition}
$(L, \cdot)\in\setH$.
\end{proposition}
Let $G$ be a group such that $G\cong L$ as sets
and the map
$\pi: L\to G$ means a set-theoretic bijection.
Let $\mu_1^G: G\times G\times G\to G$
denote the ternary operation
of $G$
defined by
\begin{equation}
\label{eq:mu1G}
\mu_1^G(a, b, c)=
ab^{-1}c
\quad
(a, b, c\in G).
\end{equation}
\begin{remark}
This $\mu_1^G: G\times G\times G\to G$ is a (classical) torsor
\cite[Remark 5.2]{shibukawa}
(For torsors, see \cite[Section 4.2]{kontsevich}
and
\cite[Section 1]{skoda}).
\end{remark}
The maps $\xi_\lambda: L\times L\to L$ and
$\eta_\lambda: L\times L\to L$ $(\lambda\in H)$ are given by
\begin{align}
&\xi_\lambda(a, b)=\lambda\backslash\pi^{-1}(\mu_1^G(\pi(\lambda),
\pi(\lambda a), \pi((\lambda a)b)));
\label{eq:xi}
\\
&\eta_\lambda(a, b)
=
(\lambda\xi_\lambda(a, b))\backslash((\lambda a)b)
\quad(a, b\in L).
\label{eq:eta}
\end{align}
For $\lambda\in H, a, b\in L$, we define the map
$\sigma(\lambda): L\times L\to L\times L$
by
\begin{equation}\label{eq:sigma}
\sigma(\lambda)(a, b)
=
(\xi_\lambda(a, b), \eta_\lambda(a, b)).
\end{equation}

This $\sigma$ is a dynamical Yang-Baxter map;
that is to say,
\begin{proposition}\label{prop:sigmabraid}
$\sigma: L\otimes L\to L\otimes L$
is a solution to the braid relation $\eqref{eq:braidrel}$
in $\setH$.
\end{proposition}
Because
the ternary operation
$\mu_1^G$
satisfies
\begin{align*}
&\mu_1^G(a, \mu_1^G(a, b, c), \mu_1^G(\mu_1^G(a, b, c), c, d))
=\mu_1^G(a, b, \mu_1^G(b, c, d)),
\\
&\mu_1^G(\mu_1^G(a, b, c), c, d)=
\mu_1^G(\mu_1^G(a, b, \mu_1^G(b, c, d)), \mu_1^G(b, c, d), d)
\end{align*}
for all $a, b, c, d\in G$,
the morphism $\sigma$ 
\eqref{eq:sigma}
satisfies the braid relation $\eqref{eq:braidrel}$
in $\setH$
(See \cite[Theorem 3.3]{matsumoto2} and \cite[Theorem 3.2]{shibu2}).
\begin{remark}
This solution is a dynamical braiding map constructed by
$(G, \mu_1^G)$ in \cite[Section 6]{shibu2}
for the group $G$
(See also \cite[Remark 6.7]{shibu2}).
\end{remark}
\begin{proposition}\label{prop:sigmaisom}
$\sigma: L\otimes L\to L\otimes L$
is an isomorphism
in $\setH$.
\end{proposition}
The inverse $\sigma^{-1}: L\otimes L\to L\otimes L$ is defined by
$\sigma^{-1}(\lambda)(a, b)
=(\lambda\backslash c,
c\backslash
((\lambda a)b))$,
where
$c=\pi^{-1}(\pi((\lambda a)b)\pi(\lambda a)^{-1}\pi(\lambda))$
(Cf.\ \cite[Proposition 5.1]{shibu1}).

We set 
\begin{equation}
\label{eq:m}
m(\lambda)(a, b)=\lambda\backslash((\lambda a)b);
\eta(\lambda)(\bullet)=e_L
\quad(\lambda\in H, a, b\in L, I=\{\bullet\}).
\end{equation}
\begin{proposition}\label{prop:braidedmonoid}
$(L, m, \eta, \sigma)$ is a braided monoid in $\setH$
$($For the definition of the braided monoid, see $\mathrm{Definition}$
$\ref{def:braided}$$)$.
Moreover, $m\sigma=m$
and the
map $L\ni b\mapsto m(\lambda)(a, b)\in L$ is injective for all
$\lambda\in H$ and $a\in L$.
\end{proposition}
\begin{proof}

We only prove \eqref{eq:braidedmonoid2}.
We show
\begin{align}
&(m\otimes 1_L)(\lambda)a_{LLL}^{-1}(\lambda)(1_L\otimes\sigma)(\lambda)
a_{LLL}(\lambda)(\sigma\otimes 1_L)(\lambda)((a, b), c)
\label{eq:prbraidmon2}
\\
=&
\sigma(\lambda)(1_L\otimes m)(\lambda)a_{LLL}(\lambda)((a, b), c)
\notag
\end{align}
for all $\lambda\in H(=L)$ and $a, b, c\in L$.

Because of \eqref{eq:sigma},
the left-hand-side of \eqref{eq:prbraidmon2}
is
\[
(\lambda\backslash((\lambda\xi_\lambda(a, b))\xi_{\lambda\xi_\lambda(a, b)}
(\eta_\lambda(a, b), c)),
\eta_{\lambda\xi_\lambda(a, b)}(\eta_\lambda(a, b), c)),
\]
which is exactly the same as
\begin{align*}
&(\lambda\backslash\pi^{-1}(\pi(\lambda\xi_\lambda(a, b))
\pi((\lambda\xi_\lambda(a, b))\eta_\lambda(a, b))^{-1}
\pi(((\lambda\xi_\lambda(a, b))\eta_\lambda(a, b))c)),
\\
&\quad
((\lambda\xi_\lambda(a, b))\xi_{\lambda\xi_\lambda(a, b)}(\eta_\lambda(a, b),
c))\backslash(((\lambda\xi_\lambda(a, b))\eta_\lambda(a, b))c))
\\
=&
(\lambda\backslash\pi^{-1}(\pi(\lambda)\pi(\lambda a)^{-1}\pi(((\lambda a)b)c)),
\pi^{-1}(\pi(\lambda)\pi(\lambda a)^{-1}\pi(((\lambda a)b)c))\backslash
(((\lambda a)b)c))
\end{align*}
in view of \eqref{eq:xi}
and
\eqref{eq:eta}.

Taking account of \eqref{eq:sigma} again,
we see that
the right-hand-side of \eqref{eq:prbraidmon2}
is
\[
(\xi_\lambda(a, (\lambda a)\backslash(((\lambda a)b)c)),
\eta_\lambda(a, (\lambda a)\backslash(((\lambda a)b)c))),
\]
which coincides with
\[
(\lambda\backslash\pi^{-1}(\pi(\lambda)\pi(\lambda a)^{-1}\pi(((\lambda a)b)c)),
\pi^{-1}(\pi(\lambda)\pi(\lambda a)^{-1}\pi(((\lambda a)b)c))\backslash
(((\lambda a)b)c))
\]
by virtue of \eqref{eq:xi}
and
\eqref{eq:eta}.
Hence \eqref{eq:prbraidmon2} holds.
%
%
\end{proof}

For $\lambda\in H$, the map $\iota^\lambda: L\to L\times L$
is given by
\begin{equation}\label{eq:iota}
\iota^\lambda(a)=(a, (\lambda a)\backslash\lambda)
\quad(a\in L).
\end{equation}
Let $(X, m_X)$ be a left $(L, m, \eta)$-module in $\setH$.
We set $Y=L\otimes X\in\setH$.
Let $\lambda_0(\in L)$ denote the unique element satisfying that
the element $\pi(\lambda_0)$ is the unit of the group $G$.

We can now formulate our main result
whose proof will be given 
in the next section.
\begin{theorem}\label{theorem:grouphomo1}
$(1)$
Let $(Y, m_Y)$ be a left $(L, m, \eta)$-module in $\setH$
satisfying
$\eqref{eq:mXmY}$
and that two pairs
$(m_Y, m_Y^{\mathrm{triv}})$ and
$(m_Y, m_Y^{\sigma})$ braid-commute respectively.
We define the morphism $\theta_Y: (L\otimes L)\otimes Y\to Y(\in\setH)$
by $\eqref{eq:theta}$.
For $\lambda\in H, a, b\in L, x\in X$, let 
$a\shikaku{\lambda}{x}b(\in L)$
denote the first component of $\theta_Y(\lambda)(\iota^\lambda(a),
(b, m_X(\lambda b)((\lambda b)\backslash\lambda, x)))$.
The map $f^{\lambda_0}_x: G\to G$
$(x\in X)$
is given by
\[
f^{\lambda_0}_x(a)=
\pi(\lambda_0)\pi(\lambda_0
((\lambda_0\backslash
\pi^{-1}(a))
\shikaku{\lambda_0}{x}e_L))^{-1}
a
\quad(a\in G).
\]
Then
$\{ f^{\lambda_0}_x: G\to G \mid x\in X\}$
is a family
of group homomorphisms.

\noindent{$(2)$}
Let 
$\{ f^{\lambda_0}_x: G\to G \mid x\in X\}$
be a family of group homomorphisms.
For $\lambda\in H, a, b\in L, x\in X$,
we define\/$:$
\begin{align}
\notag
&a\shikaku{\lambda}{x}b
=
\lambda\backslash
\pi^{-1}(\pi(\lambda a)\pi(\lambda)^{-1}\pi(\lambda b)
f^{\lambda_0}_{m_X(\lambda_0)(\lambda_0\backslash\lambda, x)}(\pi(\lambda))
\\
\notag
&\qquad\qquad
f^{\lambda_0}_{m_X(\lambda_0)(\lambda_0\backslash\lambda, x)}
(\pi(\lambda a))^{-1});
\\
&m_Y(\lambda)(a, (b, x))
=
(\lambda\backslash c,
m_X(c)
(c\backslash(\lambda a),
m_X(\lambda a)(b, x)).
\label{eq:mY}
\end{align}
Here,
$c=(\lambda a)(((\lambda a)\backslash\lambda)
\shikaku{\lambda a}{m_X(\lambda a)(b, x)}b)\in L$.
Then $m_Y: L\otimes Y\to Y$ 
is a morphism of $\setH$.
In addition,
$(Y, m_Y)$ is a left $(L, m, \eta)$-module in $\setH$
satisfying
$\eqref{eq:mXmY}$
and that two pairs
$(m_Y, m_Y^{\mathrm{triv}})$ and
$(m_Y, m_Y^{\sigma})$ braid-commute respectively.

\noindent{$(3)$}
The correspondence in $(1)$
is the inverse of that in $(2)$
and vice versa.
\end{theorem}

It follows from 
the above theorem,
together with Propositions
\ref{prop:k1}, \ref{prop:k2}, \ref{prop:sigmabraid},
\ref{prop:sigmaisom}
and
Remarks \ref{rem:braidcommmYmYtriv},
\ref{rem:braidcommmYmYsigma},
that 
Corollary \ref{thm:reflection} and Proposition \ref{cor:reflection}
can produce dynamical reflection maps 
$k: L\otimes X\to L\otimes X\in\setH$
\eqref{eq:k}.
\section{Proof of Theorem \ref{theorem:grouphomo1}}
\label{section:proof}
This section is devoted to giving a proof of
Theorem \ref{theorem:grouphomo1}
(Cf.\ \cite[Section 8]{decommer}).

Let $L=(L, \cdot, e_L)$ be a left quasigroup with a unit
(Definition \ref{def:leftquasig}).
We set $H=L$.
Let $G$ be a group isomorphic to $L$ as sets.
By Proposition
\ref{prop:braidedmonoid},
$(L, m, \eta, \sigma)$
is a braided monoid in $\setH$
(see Definition \ref{def:braided}, \eqref{eq:sigma}, 
and \eqref{eq:m}).

Let $(X, m_X)$ be a left $(L, m, \eta)$-module in $\setH$.
We write
$Y=L\otimes X\in\setH$.
For $\lambda\in H, a, b\in L$,
elements $a\underset{\lambda}{\cdot}b$ and $\rho^\lambda_b(a)\in L$
are defined by:
\begin{align}
&a\underset{\lambda}{\cdot}b=\lambda\backslash
\pi^{-1}(\mu_1^G(\pi(\lambda a),
\pi(\lambda), \pi(\lambda b)));
\label{eq:dotlambda}
\\
&\rho^\lambda_b(a)=\lambda\backslash
\pi^{-1}(\mu_1^G(\pi((\lambda b)a), \pi(\lambda b), \pi(\lambda))).
\label{eq:rholambda}
\end{align}
Here,
$\mu_1^G: G\times G\times G\ni(a, b, c)\mapsto ab^{-1}c\in G$
is the ternary operation in
\eqref{eq:mu1G}.
\begin{remark}\label{rem:dotlambda}
$L$ with the binary operation \eqref{eq:dotlambda}
is a group
(Cf.\ \cite[Proposition 4.11]{shibu1}).
Its unit element is $e_L$
and
the inverse of $a\in L$ is
$\lambda\backslash\pi^{-1}(\mu_1^G(\pi(\lambda),
\pi(\lambda a),
\pi(\lambda)))$
(See \cite[Section 1]{johnstone}
and
\cite[Remark 5.3]{shibukawa}).
\end{remark}
\begin{proposition}\label{prop:thetaYshikaku1}
Let $(Y, \theta_Y)$ be a left module
of the twisted monoid $L\underset{\mathrm{tw}}{\otimes}L$
satisfying $\eqref{eq:thetamodule1}$
and $\eqref{eq:thetamodule2}$
$($For the twisted monoid, see below $\mathrm{Proposition}$
$\ref{prop:twistedmonnoid}$$)$.
We define
$a\shikaku{\lambda}{x}b(\in L)$
$(\lambda\in H, a, b\in L, x\in X)$
by
the first component of $\theta_Y(\lambda)(\iota^\lambda(a),
(b, m_X(\lambda b)((\lambda b)\backslash\lambda, x)))$
$($For 
$\iota^\lambda(a)$, see
$\eqref{eq:iota}$$)$.
Then they enjoy the following for all $\lambda\in H, a, b, c\in L$, 
and $x\in X$\/$:$
\begin{align}
&(a\underset{\lambda}{\cdot}b)\shikaku{\lambda}{x}c
=a\shikaku{\lambda}{x}(b\shikaku{\lambda}{x}c);
\label{eq:shikaku1}
\\
&e_L\shikaku{\lambda}{x}b=b;
\label{eq:shikaku2}
\\
&\lambda\backslash((\lambda b)(a\shikaku{\lambda b}{x}c))
=\rho^\lambda_b(a)\shikaku{\lambda}{m_X(\lambda)(b, x)}(\lambda\backslash
((\lambda b)c)).
\label{eq:shikaku3}
\end{align}
\end{proposition}
\begin{proof}
We first prove \eqref{eq:shikaku1}
and \eqref{eq:shikaku2}
by means of the following claim.
\begin{claim}\label{clm:theta}
For $\lambda\in H, a, b\in L$,
and $x\in X$,
\[
\theta_Y(\lambda)(\iota^\lambda(a), (b, m_X(\lambda b)((\lambda b)\backslash
\lambda, x)))
=
(a\shikaku{\lambda}{x}b,
m_X(\lambda(a\shikaku{\lambda}{x}b))
((\lambda(a\shikaku{\lambda}{x}b))\backslash\lambda, x)).
\]
\end{claim}
Assuming this claim for the moment, we complete the proof
of \eqref{eq:shikaku1}
and \eqref{eq:shikaku2}.

It follows from the above claim 
and the definition of 
$(a\underset{\lambda}{\cdot}b)\shikaku{\lambda}{x}c$
that,
for $\lambda\in H, a, b, c\in L$,
and $x\in X$,
\begin{align*}
&((a\underset{\lambda}{\cdot}b)\shikaku{\lambda}{x}c,
m_X(\lambda((a\underset{\lambda}{\cdot}b)\shikaku{\lambda}{x}c))
((\lambda((a\underset{\lambda}{\cdot}b)\shikaku{\lambda}{x}c))\backslash
\lambda, x))
\\
=&
\theta_Y(\lambda)(\iota^\lambda(a\underset{\lambda}{\cdot}b),
(c, m_X(\lambda c)((\lambda c)\backslash\lambda, x))),
\end{align*}
which is
\begin{equation}
\label{eq:prop711}
\theta_Y(\lambda)(m_{L\otimes L}(\lambda)
(\iota^\lambda(a), \iota^\lambda(b)),
(c, m_X(\lambda c)((\lambda c)\backslash\lambda, x))),
\end{equation}
since
\eqref{eq:mAotimesA1}, \eqref{eq:xi},
\eqref{eq:eta},
\eqref{eq:m},
\eqref{eq:iota},
and
\eqref{eq:dotlambda}
induce
\begin{equation}
\label{eq:prop718}
\iota^\lambda(a\underset{\lambda}{\cdot}b)=
m_{L\otimes L}(\lambda)
(\iota^\lambda(a), \iota^\lambda(b)).
\end{equation}
Because 
$(Y, \theta_Y)$ is a left module of the twisted monoid
$L\underset{\mathrm{tw}}{\otimes}L$
and
$\lambda\iota^\lambda(a)=\lambda$,
\eqref{eq:prop711}
is
\begin{equation}
\label{eq:prop712}
\theta_Y(\lambda)(\iota^\lambda(a),
\theta_Y(\lambda)(\iota^\lambda(b),
(c, m_X(\lambda c)((\lambda c)\backslash\lambda, x)))).
\end{equation}
By using Claim \ref{clm:theta} again,
we can show that \eqref{eq:prop712} coincides with
$(a\shikaku{\lambda}{x}(b\shikaku{\lambda}{x}c),$
$m_X(\lambda(a\shikaku{\lambda}{x}(b\shikaku{\lambda}{x}c)))
((\lambda(a\shikaku{\lambda}{x}(b\shikaku{\lambda}{x}c)))\backslash
\lambda, x))$,
and we have thus proved \eqref{eq:shikaku1},
comparing the first components.

The next task is to show \eqref{eq:shikaku2} with the aid of Claim
\ref{clm:theta}.
From this claim
and \eqref{eq:etaAotimesA1},
\begin{align}
&(e_L\shikaku{\lambda}{x}b,
m_X(\lambda(e_L\shikaku{\lambda}{x}b))
((\lambda(e_L\shikaku{\lambda}{x}b))\backslash\lambda, x))
\label{eq:prop715}
\\
=&
\theta_Y(\lambda)((e_L, e_L), (b, 
m_X(\lambda b)((\lambda b)\backslash\lambda, x))
\notag
\\
=&
(\theta_Y(\eta_{L\otimes L}\otimes 1_Y))(\lambda)
(\bullet, (b, 
m_X(\lambda b)((\lambda b)\backslash\lambda, x))).
\notag
\end{align}
Since
$(Y, \theta_Y)$ is a left module of the twisted monoid
$L\underset{\mathrm{tw}}{\otimes}L$,
the right-hand-side of
\eqref{eq:prop715} coincides with
$l_{L\otimes L}(\lambda)(\bullet, (b, 
m_X(\lambda b)((\lambda b)\backslash\lambda, x)))
=
(b, 
m_X(\lambda b)((\lambda b)\backslash\lambda,
x))$,
which establishes \eqref{eq:shikaku2}.

The task is now to prove \eqref{eq:shikaku3}.
For the proof, we need the following.
\begin{claim}\label{clm:mLotimesL}
For $\lambda\in H, a, b, c\in L$,
and $x\in X$,
\begin{align}
m_{L \otimes L}(\lambda)((b, e_{L}),  \iota^{\lambda b} (a))
&=m_{L \otimes L}(\lambda)
 (\iota^{\lambda} (\rho^\lambda_b(a)), (b, e_{L}))
 \label{eq:pfiotaandrho}
 \\
&=
(\lambda\backslash((\lambda b)a), ((\lambda b)a)\backslash(\lambda b)), 
 \notag
 \end{align}
 \begin{align}
 &\theta_{Y}(\lambda)((b, e_{L}), (c, m_{X}((\lambda b) c)(((\lambda b) c) \backslash (\lambda b), x))) 
\label{eq:pfthetaYandshikaku1}  \\
= &(\lambda \backslash ((\lambda b) c), 
m_{X}((\lambda b) c)(((\lambda b) c) \backslash \lambda, 
m_{X}(\lambda)(b, x))).
\notag
\end{align}
\end{claim}
Assuming this claim for the moment, we complete the proof
of \eqref{eq:shikaku3}.
From \eqref{eq:pfiotaandrho},
\begin{align}
&\theta_Y(\lambda)
(m_{L \otimes L}(\lambda)(\iota^\lambda(\rho^\lambda_b(a)), (b, e_{L})), 
(c, m_{X}((\lambda b) c)(((\lambda b) c) \backslash (\lambda b), x)))
\label{eq:prop714}
\\
=&\theta_Y(\lambda)
(m_{L \otimes L}(\lambda)((b, e_{L}), \iota^{\lambda b}(a)), 
(c, m_{X}((\lambda b) c)(((\lambda b) c) \backslash (\lambda b), x))),\notag
\end{align}
and we compute the both sides of \eqref{eq:prop714}.

Because of \eqref{eq:pfthetaYandshikaku1}, 
$\lambda(\iota^\lambda(\rho^\lambda_b(a)))=\lambda$, and 
the fact that $(Y, \theta_Y)$ is a left module of the twisted monoid
$L\underset{\mathrm{tw}}{\otimes}L$,
the left-hand-side of \eqref{eq:prop714}
is
\begin{align*}
&\theta_Y(\lambda)(\iota^\lambda(\rho^\lambda_b(a)), 
\theta_Y(\lambda\iota^\lambda(\rho^\lambda_b(a)))
((b, e_{L}), (c, m_{X}((\lambda b) c)(((\lambda b) c) \backslash (\lambda b), x))))
\\
=&
\theta_Y(\lambda)(\iota^\lambda(\rho^\lambda_b(a)), 
(\lambda\backslash((\lambda b)c), 
m_{X}((\lambda b) c)(((\lambda b) c) \backslash \lambda, m_X(\lambda)
(b, x)))).
\end{align*}
By the definition,
the first component of the right-hand-side of the above equation is exactly
$\rho^\lambda_b(a)\shikaku{\lambda}{m_X(\lambda)(b, x)}(\lambda\backslash
((\lambda b)c))$.

A slightly change in the proof actually shows that
the first component of the right-hand-side of \eqref{eq:prop714}
is 
$\lambda\backslash((\lambda b)(a\shikaku{\lambda b}{x}c))$,
and the proof of \eqref{eq:shikaku3} is therefore complete.
\end{proof}
\begin{proof}[Proof of Claim $\ref{clm:theta}$]
By the definition, 
$a\shikaku{\lambda}{x}b(\in L)$
$(\lambda\in H, a, b\in L, x\in X)$
is
the first component of $\theta_Y(\lambda)(\iota^\lambda(a),
(b, m_X(\lambda b)((\lambda b)\backslash\lambda, x)))$,
and we define $y\in X$ by
\begin{equation}
\label{eq:claim2}
\theta_Y(\lambda)(\iota^\lambda(a), (b, m_X(\lambda b)((\lambda b)\backslash
\lambda, x)))
=
(a\shikaku{\lambda}{x}b, y).
\end{equation}
Our goal is to show 
\begin{equation}
\label{eq:claim3}
y=
m_X(\lambda(a\shikaku{\lambda}{x}b))
((\lambda(a\shikaku{\lambda}{x}b))\backslash\lambda, x).
\end{equation}

From \eqref{eq:thetamodule1},
\begin{align}
&(m_X\theta_Y)(\lambda)(\iota^\lambda(a), (b, m_X(\lambda b)((\lambda b)\backslash
\lambda, x)))
\label{eq:claim1}
\\
=&
m_X(\lambda)(1_L\otimes m_X)(\lambda)a_{LLX}(\lambda)
(\iota^\lambda(a), m_X(\lambda)(b, m_X(\lambda b)((\lambda b)\backslash
\lambda, x))).
\notag
\end{align}
On account of \eqref{eq:m}
and the fact that
$(X, m_X)$ is a left $(L, m, \eta)$-module,
\begin{align}
&m_X(\lambda)(b, m_X(\lambda b)((\lambda b)\backslash
\lambda, x))
\label{eq:prop716}
\\
=&
(m_X(m\otimes 1_X)a_{LLX}^{-1})(\lambda)(b, ((\lambda b)\backslash
\lambda, x))
\notag
\\
=&
m_X(\lambda)(e_L, x)
\notag
\\
=&
(m_X(\eta\otimes 1_X))(\lambda)(\bullet, x)
\notag
\\
=&
l_X(\lambda)(\bullet, x)
\notag
\\
=&x,
\notag
\end{align}
and
the right-hand-side of \eqref{eq:claim1} is consequently
$m_X(\lambda)(a, m_X(\lambda a)((\lambda a)\backslash\lambda, x))$,
which coincides with
$x$
due to \eqref{eq:prop716}.
In view of \eqref{eq:claim2},
the left-hand-side of \eqref{eq:claim1} is
$m_X(\lambda)(a\shikaku{\lambda}{x}b, y)$,
and 
$x=m_X(\lambda)(a\shikaku{\lambda}{x}b, y)$
as a result.

The right-hand-side of \eqref{eq:claim3} is hence
\begin{align*}
&m_X(\lambda(a\shikaku{\lambda}{x}b))
((\lambda(a\shikaku{\lambda}{x}b))\backslash\lambda,
m_X(\lambda)(a\shikaku{\lambda}{x}b, y))
\\
=&
(m_X(1_L\otimes m_X))(\lambda(a\shikaku{\lambda}{x}b))
((\lambda(a\shikaku{\lambda}{x}b))\backslash\lambda,
(a\shikaku{\lambda}{x}b, y)),
\end{align*}
which coincides with
$y$,
because of 
\eqref{eq:m}
and the fact that $(X, m_X)$ is a left $(L, m, \eta)$-module.
Therefore, the claim follows. 
\end{proof}
\begin{proof}[Proof of Claim $\ref{clm:mLotimesL}$]
On account of \eqref{eq:thetamodule2}, we can  prove \eqref{eq:pfthetaYandshikaku1} immediately.
Let us show \eqref{eq:pfiotaandrho}.
Because of \eqref{eq:mAotimesA1},
\eqref{eq:xi},
\eqref{eq:eta},
\eqref{eq:m},
and \eqref{eq:iota},
the right-hand-side of \eqref{eq:pfiotaandrho} is
\begin{align*}
&m_{L \otimes L}(\lambda)((\rho^\lambda_b(a), 
(\lambda \rho^\lambda_b(a)) \backslash \lambda), (b, e_{L})) 
\\
=& (\lambda \backslash \pi^{-1}(\pi(\lambda \rho^\lambda_b(a)) 
\pi(\lambda)^{-1}\pi(\lambda b)), (\pi^{-1}(\pi(\lambda \rho^\lambda_b(a)) 
\pi(\lambda)^{-1}\pi(\lambda b))) \backslash (\lambda b)).
\end{align*}
By \eqref{eq:rholambda},
$\pi^{-1}(\pi(\lambda \rho^\lambda_b(a)) 
\pi(\lambda)^{-1}\pi(\lambda b))=(\lambda b)a$,
and
the right-hand-side of \eqref{eq:pfiotaandrho}
is $(\lambda\backslash((\lambda b) a), ((\lambda b)a)\backslash(\lambda b))$
as a result.

In the same manner, we can see that
the left-hand-side of \eqref{eq:pfiotaandrho}
is
\begin{align*}
&(m(\lambda)(b, \xi_{\lambda b}(e_L, a)),
\\
&\quad
m((\lambda b)\xi_{\lambda b}(e_L, a))(((\lambda b)\xi_{\lambda b}(e_L, a))
\backslash(((\lambda b)e_L)a), ((\lambda b)a)\backslash(\lambda b)))
\\
=& (\lambda \backslash ((\lambda b) a), ((\lambda b) a) \backslash 
(\lambda b)),
\end{align*}
and
\eqref{eq:pfiotaandrho} is thus proved.
%
\end{proof}
\begin{proposition}\label{prop:thetaYshikaku2}
Let
$a\shikaku{\lambda}{x}b$
$(\lambda\in H, a, b\in L, x\in X)$
be elements of $L$ satisfying
$\eqref{eq:shikaku1}$,
$\eqref{eq:shikaku2}$,
and
$\eqref{eq:shikaku3}$
for all $\lambda\in H, a, b, c\in L$,
and $x\in X$.
We define
$\theta_Y(\lambda)((a, b), (c, x))\in Y$
$(\lambda\in H, a, b, c\in L, x\in X)$
by\/$:$
\begin{align}
\theta_Y(\lambda)((a, b), (c, x))
&=
(\lambda\backslash d, 
m_X(d)(d\backslash((\lambda a)b), m_X((\lambda a)b)(c, x)))
\label{eq:deftheta}
\\
&\!
(=(\lambda\backslash d, m_X(d)(d\backslash(((\lambda a)b)c), x))).
\label{eq:deftheta2}
\end{align}
Here, $d=((\lambda a)b)((((\lambda a)b)\backslash(\lambda a))
\shikaku{(\lambda a)b}{m_X((\lambda a)b)(c, x)}c)$.
Then
$(Y, \theta_Y)$ is a left module
of the twisted monoid
$L\underset{\mathrm{tw}}{\otimes}L$
satisfying $\eqref{eq:thetamodule1}$
and $\eqref{eq:thetamodule2}$.
\end{proposition}
\begin{proof}
An easy computation shows that
$\lambda\theta_Y(\lambda)((a, b), (c, x))
=\lambda((a, b), (c, x))$
for
$\lambda\in H, (a, b)\in L\times L$,
and
$(c, x)\in L\times X$,
and
$\theta_{Y}: (L\otimes L)\otimes Y\to Y$ is consequently
a morphism in $\setH$. 

Because of \eqref{eq:leftmod2},
\eqref{eq:m},
\eqref{eq:shikaku2},
and the fact that
$a\backslash a
=e_L$
for every $a\in L$,
we check at once that $\theta_{Y}: (L\otimes L)\otimes Y\to Y$ 
satisfies \eqref{eq:thetamodule2}. 
From the definition \eqref{eq:m}, \eqref{eq:deftheta}, 
and
the fact that $(X, m_X)$ is a left $(L, m, \eta)$-module, we
can  show \eqref{eq:thetamodule1}.
%

On account of \eqref{eq:leftmod2},
\eqref{eq:etaAotimesA1},
and $\eqref{eq:shikaku2}$, it is immediate that 
$\theta_{Y}(\eta_{L \otimes L} \otimes 1_{Y}) = l_{Y}$. 

Now we show 
$\theta_{Y}(1_{L \otimes L} \otimes \theta_{Y}) 
= \theta_{Y}(m_{L \otimes L} \otimes 1_{Y})a_{L \otimes L L \otimes L Y}^{-1}$. 
For the proof,
we need:
\begin{claim}
\label{clm:claim5}
For $\lambda\in H(=L), a, b, c, d, f\in L$, and $x\in X$,
\begin{align}
&
\lambda\backslash
((\lambda a)(b\shikaku{\lambda a}{m_X(\lambda a)(c, x)}((\lambda a)\backslash
(((\lambda a)c)(d\shikaku{(\lambda a)c}{x}f)))))
\label{eq:prop7111}
\\
=&
(\rho^\lambda_a(b)
\underset{\lambda}{\cdot}
\rho^\lambda_a(\rho^{\lambda a}_c(d)))
\shikaku{\lambda}{m_X(\lambda)(\lambda\backslash((\lambda a)c), x)}
(\lambda\backslash(((\lambda a)c)f)),
\notag
\\
&\theta_Y(\lambda)
(m_{L\otimes L}(\lambda)(\iota^\lambda(a), (b, e_L)),
(f, m_X((\lambda b)f)(((\lambda b)f)\backslash(\lambda b), x)))
\label{eq:prop7113}
\\
=&
(a\shikaku{\lambda}{m_X(\lambda)(b, x)}(\lambda\backslash((\lambda b)f)), m_X(\lambda(a\shikaku{\lambda}{m_X(\lambda)(b, x)}
(\lambda\backslash((\lambda b)f)))) 
\notag
\\
&\quad(\lambda(a\shikaku{\lambda}{m_X(\lambda)(b, x)}
(\lambda\backslash((\lambda b)f)))\backslash(\lambda b), x)).
\notag
\end{align}
\end{claim}
\begin{claim}\label{clm:claim6}
For every $\lambda\in H(=L)$,
the following map
$F^\lambda: (L\times L)\times((L\times L)\times Y)\to
(L\times L)\times((L\times L)\times Y)$
is bijective\/$:$
\begin{align*}
&F^\lambda((a, b), ((c, d), (f, x)))
\\
=&
(m_{L\otimes L}(\lambda)((a, e_L), \iota^{\lambda a}(b)),
(m_{L\otimes L}(\lambda a)((c, e_L), \iota^{(\lambda a)c}(d)),
\\
&\quad(f, m_X(((\lambda a)c)f)((((\lambda a)c)f)\backslash((\lambda a)c), x)))).
\end{align*}
Here,
$((a, b), ((c, d), (f, x)))\in(L\times L)\times((L\times L)\times Y)$.
\end{claim}
Assuming these claims for the moment, we complete the proof.

On account of Claim \ref{clm:claim6},
it suffices to show
\begin{align}
&(\theta_{Y}(1_{L \otimes L} \otimes \theta_{Y}))(\lambda)
F^\lambda((a, b), ((c, d), (f, x)))
\label{eq:prop7112}
\\
= &(\theta_{Y}(m_{L \otimes L} \otimes 1_{Y})a_{L \otimes L L \otimes L Y}^{-1})
(\lambda)F^\lambda((a, b), ((c, d), (f, x)))
\notag
\end{align}
for $\lambda\in H, a, b, c, d, f\in L$, and $x\in X$.

Because of \eqref{eq:m}
and the fact that $(X, m_X)$ is a left $(L, m, \eta)$-module,
\[
m_X((\lambda a)c)(f, m_X(((\lambda a)c)f)((((\lambda a)c)f)\backslash((\lambda a)c),
x))=x,
\]
and consequently
the left-hand-side of \eqref{eq:prop7112} is
\begin{align*}
&\theta_Y(\lambda)((\lambda\backslash((\lambda a)b), ((\lambda a)b)\backslash
(\lambda a)),
\theta_Y(\lambda a)
(((\lambda a)
\backslash(((\lambda a)c)d),
\\
&\quad(((\lambda a)c)d)\backslash((\lambda a)c)),(f, m_X(((\lambda a)c)f)((((\lambda a)c)f)\backslash((\lambda a)c), x))))
\\
=&
\theta_Y(\lambda)((\lambda\backslash((\lambda a)b), ((\lambda a)b)\backslash
(\lambda a)),
((\lambda a)\backslash(((\lambda a)c)
(d\shikaku{(\lambda a)c}{x}f)),
\\
&\quad
m_X(((\lambda a)c)
(d\shikaku{(\lambda a)c}{x}f))
((((\lambda a)c)(d\shikaku{(\lambda a)c}{x}f)) \backslash ((\lambda a)c), x)))
\end{align*}
in view of \eqref{eq:pfiotaandrho}
and \eqref{eq:deftheta}.

By using 
\eqref{eq:mAotimesA}
for $A=L$,
\eqref{eq:deftheta},
and the fact that
$(X, m_X)$ is a left $(L, m, \eta)$-module,
we deduce that
the right-hand-side of the above equation is
\begin{equation}
\label{eq:prop7115}
(\lambda\backslash p, m_X(p)(p\backslash((\lambda a)c), x)).
\end{equation}
Here,
$p = (\lambda a)(b\shikaku{\lambda a}
{m_{X}(\lambda a)(c, x)}((\lambda a)
\backslash(((\lambda a)c)(d\shikaku{(\lambda a)c}{x}f))))$.

Because of 
\eqref{eq:pfiotaandrho}
and
the fact that
$\lambda\iota^\lambda(\rho^\lambda_a(b))=
(\lambda\rho^\lambda_a(b))((\lambda\rho^\lambda_a(b))\backslash\lambda)
=\lambda$
and
$\lambda(a, e_L)=\lambda a$,
the
right-hand-side of \eqref{eq:prop7112} is
\begin{align}
&\theta_Y(\lambda)(m_{L\otimes L}(\lambda)(\iota^\lambda(\rho^\lambda_a(b)),
m_{L\otimes L}(\lambda)(
m_{L\otimes L}(\lambda)((a, e_L), \iota^{\lambda a}(\rho^{\lambda a}_c(d))),
\label{eq:prop7114}
\\
&\quad 
(c, e_L))),
(f,m_X(((\lambda a)c)f)((((\lambda a)c)f)\backslash((\lambda a)c), x)))
\notag
\\
=&
\theta_Y(\lambda)(m_{L\otimes L}(\lambda)(\iota^\lambda(\rho^\lambda_a(b)),
m_{L\otimes L}(\lambda)(
m_{L\otimes L}(\lambda)(\iota^\lambda(\rho^\lambda_a(\rho^{\lambda a}_c(d))), (a, e_L)),
\notag
\\
&\quad
 (c, e_L))),
(f,m_X(((\lambda a)c)f)((((\lambda a)c)f)\backslash((\lambda a)c), x))).
\notag
\end{align}

Owing to \eqref{eq:mAotimesA1}
and
\eqref{eq:xi}--\eqref{eq:sigma},
\begin{equation}
m_{L\otimes L}(\lambda)
((a, e_L), (c, e_L))=(\lambda\backslash((\lambda a)c), e_L).
\label{eq:prop7110}
\end{equation}
By taking account of 
\eqref{eq:mAotimesA}
for $A=L$,
\eqref{eq:prop718},
\eqref{eq:prop7110},
and the fact that
$\lambda(\rho^\lambda_a(\rho^{\lambda a}_c(d)))
=\lambda\iota^\lambda(\rho^\lambda_a(b))=\lambda$,
we obtain
\begin{align*}
&m_{L\otimes L}(\lambda)(\iota^\lambda(\rho^\lambda_a(b)),
m_{L\otimes L}(\lambda)(
m_{L\otimes L}(\lambda)(\iota^\lambda(\rho^\lambda_a(\rho^{\lambda a}_c(d))),
(a, e_L)), (c, e_L)))
\\
=&
m_{L\otimes L}(\lambda)(\iota^\lambda(\rho^\lambda_a(b)),
m_{L\otimes L}(\lambda)(
\iota^\lambda(\rho^\lambda_a(\rho^{\lambda a}_c(d))),
m_{L\otimes L}(\lambda)((a, e_L)), (c, e_L)))
\\
=&
m_{L\otimes L}(\lambda)(\iota^\lambda(\rho^\lambda_a(b)),
m_{L\otimes L}(\lambda)(
\iota^\lambda(\rho^\lambda_a(\rho^{\lambda a}_c(d))),
(\lambda\backslash((\lambda a)c), e_L)))
\\
=&
m_{L\otimes L}(\lambda)(m_{L\otimes L}(\lambda)(
\iota^\lambda(\rho^\lambda_a(b)),
\iota^\lambda(\rho^\lambda_a(\rho^{\lambda a}_c(d)))),
(\lambda\backslash((\lambda a)c), e_L)))
\\
=&
m_{L\otimes L}(\lambda)(
\iota^\lambda(\rho^\lambda_a(b)\underset{\lambda}{\cdot}
\rho^\lambda_a(\rho^{\lambda a}_c(d))),
(\lambda\backslash((\lambda a)c), e_L)).
\end{align*}
From
\eqref{eq:prop7111}
and
\eqref{eq:prop7113},
the right-hand-side of 
\eqref{eq:prop7114}
is
exactly the same as
\eqref{eq:prop7115},
and the proposition follows.
\end{proof}
\begin{proof}[Proof of Claims $\ref{clm:claim5}$
and $\ref{clm:claim6}$]
We first prove \eqref{eq:prop7111}
in Claim \ref{clm:claim5}.
Because of \eqref{eq:m},
\eqref{eq:shikaku3},
and
the fact that
$(X, m_X)$ is a left $(L, m, \eta)$-module,
the left-hand-side of 
\eqref{eq:prop7111}
is
\[
\rho^\lambda_a(b)\shikaku{\lambda}
{m_X(\lambda)(\lambda\backslash((\lambda a)c), x)}
(\rho^\lambda_{\lambda\backslash((\lambda a)c)}(d)
\shikaku{\lambda}{m_X(\lambda)(\lambda\backslash((\lambda a)c), x)}
(\lambda\backslash(((\lambda a)c)f))),
\]
which coincides with
\[
(\rho^\lambda_a(b)\underset{\lambda}{\cdot}
\rho^\lambda_{\lambda\backslash((\lambda a)c)}(d))
\shikaku{\lambda}{m_X(\lambda)(\lambda\backslash((\lambda a)c), x)}
(\lambda\backslash(((\lambda a)c)f)))
\]
in view of 
\eqref{eq:shikaku1}.
From
\eqref{eq:rholambda},
$\rho^\lambda_a(\rho^{\lambda a}_c(d))
=\rho^\lambda_{\lambda\backslash((\lambda a)c)}(d)$,
and
the above element is exactly the right-hand-side of
\eqref{eq:prop7111} as a result.

The next task is to show
\eqref{eq:prop7113}.
From
\eqref{eq:mAotimesA1}
and
\eqref{eq:xi}--\eqref{eq:sigma},
the left-hand-side of \eqref{eq:prop7113}
is 
\begin{align*}
&\theta_Y(\lambda)((\lambda\backslash((\lambda a)\xi_{\lambda a}((\lambda a)
\backslash\lambda ,b)),
m((\lambda a)\xi_{\lambda a}((\lambda a)\backslash\lambda,
b))
(\eta_{\lambda a}((\lambda a)\backslash\lambda, b), e_L)), 
\\
&\quad
(f, m_X((\lambda b)f)(((\lambda b)f)\backslash(\lambda b), x))),
\end{align*}
which coincides with
\begin{align}
\label{eq:prop7116}
(\lambda\backslash g, m_X(g)(g\backslash((\lambda b)f),
m_X((\lambda b)f)(((\lambda b)f)\backslash(\lambda b), x)))
\end{align}
because of \eqref{eq:deftheta}
and the fact that
$m_X(\lambda b)(f, m_X((\lambda b)f)(((\lambda b)f)\backslash(\lambda b), x))
=x$.
Here,
$g=(\lambda b)(((\lambda b)\backslash
((\lambda a)\xi_{\lambda a}((\lambda a)\backslash\lambda, b)))
\shikaku{\lambda b}{x}f)$.

On account of \eqref{eq:rholambda}
and
\eqref{eq:shikaku3},
$\lambda\backslash g=a\shikaku{\lambda}{m_X(\lambda)(b, x)}(\lambda\backslash
((\lambda b)f))$,
and
the element \eqref{eq:prop7116}
is
exactly the right-hand-side of
\eqref{eq:prop7113}
with the aid of \eqref{eq:m}
and the fact that $(X, m_X)$ is a left $(L, m, \eta)$-module.
This gives \eqref{eq:prop7113},
and the proof of Claim \ref{clm:claim5} is complete.

The map $F^\lambda$ has its inverse $G^\lambda$
defined by
\begin{align*}
&G^\lambda((a, b), ((c, d), (f, x)))
\\
=&
((\lambda\backslash((\lambda a)b),
((\lambda a)b)\backslash(\lambda a)),
((((\lambda a)b)\backslash
((((\lambda a)b)c)d),
\\
&\quad((((\lambda a)b)c)d)\backslash(((\lambda a)b)c)),(f, m_X((((\lambda a)b)c)d)(f, x)))),
\end{align*}
which proves Claim
\ref{clm:claim6}.
\end{proof}
\begin{proposition}\label{prop:thetaYshikaku3}
The correspondence in Proposition $\ref{prop:thetaYshikaku1}$
is the inverse of that in Proposition $\ref{prop:thetaYshikaku2}$
and vice versa.
\end{proposition}
\begin{proof}
We only show the correspondence in Proposition \ref{prop:thetaYshikaku2} 
is a left inverse of that in Proposition \ref{prop:thetaYshikaku1}. 
Let $(Y, \theta_Y)$ be a left module of
the twisted monoid
$L\underset{\mathrm{tw}}{\otimes}L$
satisfying \eqref{eq:thetamodule1}
and \eqref{eq:thetamodule2},
and let
$\theta'_Y: (L\otimes L)\otimes Y\to Y$ denote
the morphism of $\setH$
defined by
the right-hand-side of 
\eqref{eq:deftheta},
in which
$a\shikaku{\lambda}{x}b(\in L)$ is
the first component of $\theta_Y(\lambda)(\iota^\lambda(a),
(b, m_X(\lambda b)((\lambda b)\backslash\lambda, x)))$
as was explained in Proposition \ref{prop:thetaYshikaku1}.
We show $\theta'_Y=\theta_Y$.

For the proof,
we make use of the following claim.
\begin{claim}\label{clm:thm714}
For $\lambda\in H(=L), a, b\in L$,
and $x\in X$,
\begin{align}
&\theta'_Y(\lambda)((a, e_L), (b, x))
=\theta_Y(\lambda)((a, e_L), (b, x)),
\label{eq:prop7117}
\\
&\theta'_Y(\lambda)(\iota^\lambda(a), (b, x))
=
\theta_Y(\lambda)(\iota^\lambda(a), (b, x)).
\label{eq:prop7118}
\end{align}
\end{claim}
Assuming this claim for the moment, we complete the proof.

By
\eqref{eq:pfiotaandrho},
$m_{L\otimes L}(\lambda)
((\lambda\backslash((\lambda a)b), e_L), \iota^{(\lambda a)b}
(((\lambda a)b)\backslash(\lambda a)))=(a, b)$
for $\lambda\in H, a, b, c\in L$, and $x\in X$,
and, as a result,
\begin{align*}
&\theta'_Y(\lambda)((a, b), (c, x))
\\
=&
(\theta'_Y(m_{L\otimes L}\otimes 1_Y))(\lambda)
((\lambda\backslash((\lambda a)b), e_L), \iota^{(\lambda a)b}
(((\lambda a)b)\backslash(\lambda a)), (c, x)).
\end{align*}
From Proposition \ref{prop:thetaYshikaku2},
$(Y, \theta'_Y)$ is a left module of
the twisted monoid
$L\underset{\mathrm{tw}}{\otimes}L$,
and,
the right-hand-side of the above equation is consequently
\begin{align*}
&(\theta'_Y(1_{L\otimes L}\otimes\theta'_Y)a_{L\otimes LL\otimes LY})(\lambda)
((\lambda\backslash((\lambda a)b), e_L), \iota^{(\lambda a)b}
(((\lambda a)b)\backslash(\lambda a)), (c, x))
\\
=&
\theta'_Y(\lambda)((\lambda\backslash((\lambda a)b), e_L), 
\theta'_Y((\lambda a)b)(\iota^{(\lambda a)b}
(((\lambda a)b)\backslash(\lambda a)), (c, x))).
\end{align*}

We now apply this argument again,
with $\theta'_Y$ replaced by $\theta_Y$,
to obtain
\begin{align*}
&\theta_Y(\lambda)((a, b), (c, x))
\\
=&
\theta_Y(\lambda)((\lambda\backslash((\lambda a)b), e_L), 
\theta_Y((\lambda a)b)(\iota^{(\lambda a)b}
(((\lambda a)b)\backslash(\lambda a)), (c, x))).
\end{align*}
Claim \ref{clm:thm714}
implies that 
$\theta'_Y(\lambda)((a, b), (c, x))
=\theta_Y(\lambda)((a, b), (c, x))$,
which is the desired conclusion.
\end{proof}
\begin{proof}[Proof of Claim $\ref{clm:thm714}$]
We first prove \eqref{eq:prop7117}.
From
\eqref{eq:shikaku2}
and
\eqref{eq:deftheta},
\[
\theta'_Y(\lambda)((a, e_L), (c, x))
=(\lambda\backslash((\lambda a)c),
m_X((\lambda a)c)(e_L, x))
\]
for $\lambda\in H, a, c\in L$, and $x\in X$.
Since $(X, m_X)$ is a left $(L, m, \eta)$-module,
$m_X((\lambda a)c)(e_L, x)=x$,
and consequently
$\theta'_Y(\lambda)((a, e_L), (c, x))=(\lambda\backslash((\lambda a)c), x)$.

From
\eqref{eq:thetamodule2},
\begin{align*}
\theta_Y(\lambda)((a, e_L), (c, x))
&=
(\theta_Y((1_L\otimes\eta)\otimes 1_Y))(\lambda)
((a, \bullet), (c, x))
\\
&=
((m\otimes 1_X)a_{LLX}^{-1}(r_L\otimes 1_Y))(\lambda)
((a, \bullet), (c, x))
\\
&=
(\lambda\backslash((\lambda a)c), x),
\end{align*}
which gives \eqref{eq:prop7117}.

Let $\lambda\in H, a, b\in L$, and $x\in X$.
Because of \eqref{eq:m}
and the fact that $(X, m_X)$ is a left $(L, m, \eta)$-module,
\[
\theta_Y(\lambda)(\iota^\lambda(a), (b, x))
=\theta_Y(\iota^\lambda(a), (b, m_X(\lambda b)((\lambda b)\backslash\lambda,
m_X(\lambda)(b, x)))),
\]
which is exactly
\begin{equation}
\label{eq:prop719}
(a\shikaku{\lambda}{m_X(\lambda)(b, x)}b,
m_X(\lambda(a\shikaku{\lambda}{m_X(\lambda)(b, x)}b))
((\lambda(a\shikaku{\lambda}{m_X(\lambda)(b, x)}b))\backslash
\lambda, m_X(\lambda)(b, x)))
\end{equation}
by Claim \ref{clm:theta}.

It follows from the definition of $\theta'_Y$
\eqref{eq:deftheta2}
that
$\theta'_Y(\lambda)(\iota^\lambda(a), (b, x))$
coincides with \eqref{eq:prop719},
and \eqref{eq:prop7118} follows.
\end{proof}
We also rephrase the condition that the pair
$(m_Y, m_Y^\sigma)$ defined by $\eqref{eq:mYtensor}$
and
$\eqref{eq:mYsigma}$
braid-commutes
(Definition \ref{def:braid-commute}),
by means of the elements
$a\shikaku{\lambda}{x}b\in L$.
\begin{proposition}\label{prop:equivbraidshikaku}
Let $(Y, \theta_Y)$ be a left module
of the twisted monoid $L\underset{\mathrm{tw}}{\otimes}L$
satisfying $\eqref{eq:thetamodule1}$
and $\eqref{eq:thetamodule2}$
$($For the twisted monoid, see below $\mathrm{Proposition}$
$\ref{prop:twistedmonnoid}$$)$.
We define
$a\shikaku{\lambda}{x}b(\in L)$
$(\lambda\in H, a, b\in L, x\in X)$
by
the first component of $\theta_Y(\lambda)(\iota^\lambda(a),
(b, m_X(\lambda b)((\lambda b)\backslash\lambda, x)))$
$($For 
$\iota^\lambda(a)$, see
$\eqref{eq:iota}$$)$.
In addition,
$m_Y: L\otimes Y\to Y$
and
$m^\sigma_Y: L\otimes Y\to Y$
are defined by
$\eqref{eq:mYtensor}$
and
$\eqref{eq:mYsigma}$
for $A=(L, m, \eta)$,
respectively.
Then the following two conditions are equivalent\/$:$
\begin{enumerate}
\item[$(1)$]
$(m_Y, m_Y^{\sigma})$ braid-commutes
$($See 
$\mathrm{Definition}$ $\ref{def:braid-commute}$$)$\/$;$
\item[$(2)$]
For all $\lambda\in H, a, b, c\in L, x\in X$,
\begin{equation}\label{eq:equivbraidcomm}
a\shikaku{\lambda}{x}(b\underset{\lambda}{\cdot}c)=\lambda\backslash
\pi^{-1}(\pi(\lambda a)\pi(\lambda)^{-1}\pi(\lambda b)\pi(\lambda a)^{-1}
\pi(\lambda(a\shikaku{\lambda}{x}c))).
\end{equation}
\end{enumerate}
\end{proposition}
\begin{proof}
For $\lambda\in H(=L), a, b, c\in L$,
and $x\in X$,
we write
\begin{align*}
(f_1, y_1)=&(m_Y(1_L\otimes m_Y^{\sigma}))(\lambda)
(a, (b, (c, m_X(((\lambda a)b)c)((((\lambda a)b)c)\backslash((\lambda a)b), x)))),
\\
(f_2, y_2)=&
(m_Y^{\sigma}(1_L\otimes m_Y)a_{LLY}(\sigma\otimes 1_Y)a_{LLY}^{-1})(\lambda)
\\
&\quad(a, (b, (c, m_X(((\lambda a)b)c)((((\lambda a)b)c)\backslash((\lambda a)b), x)))).
\end{align*}
\begin{lemma}
If $f_1=f_2$,
then $y_1=y_2$.
\end{lemma}
\begin{proof}
We note that
Propositions \ref{prop:thetaYshikaku1},
\ref{prop:thetaYshikaku2},
and
\ref{prop:thetaYshikaku3}
induce \eqref{eq:deftheta}.
On account of \eqref{eq:mYtensor} and \eqref{eq:deftheta},
\begin{align}
m_Y(\lambda)(a, (b, x))
=&
(\lambda\backslash p, m_X(p)(p\backslash(\lambda a), m_X(\lambda a)(b, x)))
\label{eq:mYshikakumX}
\\
=&
(\lambda\backslash p, m_X(p)(p\backslash((\lambda a)b), x)).
\notag
\end{align}
Here, 
$p=(\lambda a)(((\lambda a)\backslash\lambda)\shikaku{\lambda a}{m_X(\lambda a)
(b, x)}b)$.
It follows from
\eqref{eq:mYsigma}
and
\eqref{eq:mYshikakumX}
that
\[
(f_1, y_1)
=
(\lambda\backslash p', 
m_X(p')(p'\backslash((\lambda a)b), x)).
\]
Here,
$p'=(\lambda a)(((\lambda a)\backslash\lambda)
\shikaku{\lambda a}{m_X(\lambda a)(b, x)}\xi_{\lambda a}(b, c))$.

We can see in a similar way that
\begin{align*}
&(f_2, y_2)
\\
=&
m_Y^\sigma(\lambda)(\xi_\lambda(a, b),
((\lambda\xi_\lambda(a, b))\backslash q, 
m_X(q)(q\backslash((\lambda a)b), x)))
\\
=&
(\xi_\lambda(\xi_\lambda(a, b), (\lambda\xi_\lambda(a, b))\backslash q),
m_X(\lambda\xi_\lambda(\xi_\lambda(a, b), (\lambda\xi_\lambda(a, b))\backslash q))
\\
&\quad
((\lambda\xi_\lambda(\xi_\lambda(a, b), (\lambda\xi_\lambda(a, b))\backslash q))
\backslash((\lambda a)b), x)).
\end{align*}
Here,
$q=
((\lambda a)b)((((\lambda a)b)\backslash(\lambda\xi_\lambda(a, b)))\shikaku{(\lambda a)b}
{x}c)$.
Hence,
$f_1=f_2$,
if and only if
\begin{equation}\label{eq:prop771}
\lambda\backslash p'
=\xi_\lambda(\xi_\lambda(a, b), (\lambda\xi_\lambda(a, b))\backslash q),
\end{equation}
which immediately induces $y_1=y_2$.
This is our assertion.
\end{proof}
From
this lemma
and the fact that
\begin{align*}
&m_X(((\lambda a)b)c)((((\lambda a)b)c)\backslash((\lambda a)b), m_X((\lambda a)b)
(c, x'))
\\
=&m_X(((\lambda a)b)c)(e_L, x')=x'
\end{align*}
for every $x'\in X$,
the condition $(1)$
in Proposition \ref{prop:equivbraidshikaku}
is equivalent to 
the following condition $(3)$:
\begin{enumerate}
\item[$(3)$]
$f_1=f_2$
for all $\lambda \in H, a, b, c\in L$, and $x\in X$.
\end{enumerate}

Because of Proposition \ref{prop:thetaYshikaku1},
\eqref{eq:shikaku3} holds, and
we can rewrite the both sides of \eqref{eq:prop771}
equivalent to the condition $(3)$
by means of
\eqref{eq:xi},
\eqref{eq:rholambda},
and \eqref{eq:shikaku3}.
\begin{align}
&\lambda\backslash
((\lambda a)(((\lambda a)\backslash\lambda)\shikaku{\lambda a}
{m_X(\lambda a)(b, x)}\xi_{\lambda a}(b, c)))
\label{eq:prop772}
\\
=&
(\lambda\backslash\pi^{-1}(\pi(\lambda)\pi(\lambda a)^{-1}\pi(\lambda)))
\shikaku{\lambda}{m_X(\lambda)(\lambda\backslash((\lambda a)b), x)} 
\notag
\\
&\quad(\lambda\backslash\pi^{-1}(\pi(\lambda a)\pi((\lambda a)b)^{-1}
\pi(((\lambda a)b)c))),
\notag
\\
\label{eq:prop773}
&\xi_\lambda(\xi_\lambda(a, b), (\lambda\xi_\lambda(a, b))\backslash
((\lambda a)b)((((\lambda a)b)\backslash(\lambda\xi_\lambda(a, b)))\shikaku{(\lambda a)b}
{x}c))
\\
=&
\lambda\backslash
\pi^{-1}(\pi(\lambda)\pi((\lambda a)b)^{-1}\pi(\lambda a)\pi(\lambda)^{-1}
\notag
\\
&\quad
\pi(\lambda((\lambda\backslash\pi^{-1}(\pi(\lambda)\pi(\lambda a)^{-1}\pi(\lambda)))
\shikaku{\lambda}{m_X(\lambda)(\lambda\backslash((\lambda a)b), x)}(\lambda
\backslash(((\lambda a)b)c))))),
\notag
\end{align}
and, as a result,
the condition $(3)$ is equivalent to the condition $(4)$:
\begin{enumerate}
\item[$(4)$]
The right-hand-side of \eqref{eq:prop772}
coincides with that of \eqref{eq:prop773} 
for all $\lambda \in H, a, b, c\in L$, and $x\in X$.
\end{enumerate}

This condition $(4)$ is equivalent to
the condition $(2)$.
We prove $(2)$ from $(4)$.
First we respectively substitute
$((\lambda a)b)\backslash(\lambda c')$ for $c$
and $m_X((\lambda a)b)
(((\lambda a)b)\backslash\lambda, x')$ for $x$
in the equation in $(4)$.
Because of \eqref{eq:m}
and the fact that
$(X, m_X)$ is a left $(L, m, \eta)$-module,
\[
m_X(\lambda)(\lambda\backslash((\lambda a)b), m_X((\lambda a)b)
(((\lambda a)b)\backslash\lambda, x')=m_X(\lambda)(e_L, x')=x',
\]
and
consequently
\begin{align}
&(\lambda\backslash\pi^{-1}(\pi(\lambda)\pi(\lambda a)^{-1}\pi(\lambda)))
\shikaku{\lambda}{x'}
(\lambda\backslash\pi^{-1}(\pi(\lambda a)\pi((\lambda a)b)^{-1}
\pi(\lambda c')))
\label{eq:prop781}
\\
=&
\lambda\backslash
\pi^{-1}(\pi(\lambda)\pi((\lambda a)b)^{-1}\pi(\lambda a)\pi(\lambda)^{-1}
\notag
\\
&\quad
\pi(\lambda((\lambda\backslash\pi^{-1}(\pi(\lambda)\pi(\lambda a)^{-1}\pi(\lambda)))
\shikaku{\lambda}{x'}c')))
\notag
\end{align}
for all $\lambda \in H, a, b, c'\in L$, and $x'\in X$.
By substituting
$(\lambda a)\backslash
\pi^{-1}(\pi(\lambda)\pi(\lambda b')^{-1}
\pi(\lambda a))$
for $b$ 
in \eqref{eq:prop781},
we deduce
that
\begin{align}
&(\lambda\backslash\pi^{-1}(\pi(\lambda)\pi(\lambda a)^{-1}\pi(\lambda)))
\shikaku{\lambda}{x}
(\lambda\backslash\pi^{-1}(\pi(\lambda b')\pi(\lambda)^{-1}
\pi(\lambda c)))
\label{eq:prop782}
\\
=&
\lambda\backslash
\pi^{-1}(\pi(\lambda)\pi(\lambda a)^{-1}\pi(\lambda b')\pi(\lambda)^{-1}
\notag
\\
&\quad
\pi(\lambda a)\pi(\lambda)^{-1}
\pi(\lambda((\lambda\backslash\pi^{-1}(\pi(\lambda)\pi(\lambda a)^{-1}\pi(\lambda)))
\shikaku{\lambda}{x}c)))
\notag
\end{align}
for all $\lambda \in H, a, b', c\in L$, and $x\in X$.
Substituting $\lambda\backslash\pi^{-1}(\pi(\lambda)
\pi(\lambda a')^{-1}\pi(\lambda))$ for $a$ in \eqref{eq:prop782}
yields 
the condition $(2)$.

We can follow the steps above in reverse,
and  the condition $(2)$
implies $(4)$.
This proves the proposition.
\end{proof}

Proposition \ref{prop:equivbraidshikaku}
immediately induces 
the following corollary on account of Theorems
\ref{thm:mYthetaY},
\ref{thm:thetaYmY},
Corollary \ref{cor:mYthetaY},
and Propositions \ref{prop:thetaYshikaku1},
\ref{prop:thetaYshikaku2},
\ref{prop:thetaYshikaku3}.
\begin{corollary}\label{cor:equivmYshikaku}
$(1)$
Let  $(Y, m_Y)$ be
a left $(L, m, \eta)$-module satisfying
$\eqref{eq:mXmY}$ and that
two pairs
$(m_Y, m_Y^{\mathrm{triv}})$ and
$(m_Y, m_Y^{\sigma})$ braid-commute respectively.
For $\lambda\in H, a, b\in L, x\in X$,
let $a\shikaku{\lambda}{x} b(\in L)$ denote
the first component of $\theta_Y(\lambda)(\iota^\lambda(a), (b,
m_X(\lambda b)((\lambda b)\backslash\lambda, x)))$
$($For 
$\theta_Y: (L\otimes L)\otimes Y\to Y\in\setH$, see
$\eqref{eq:theta}$$)$.
Then they enjoy 
$\eqref{eq:shikaku1}$,
$\eqref{eq:shikaku2}$,
$\eqref{eq:shikaku3}$,
and $\eqref{eq:equivbraidcomm}$
for all $\lambda\in H, a, b, c\in L, x\in X$.

\noindent{$(2)$}
Let
$a\shikaku{\lambda}{x}b$
$(\lambda\in H, a, b\in L, x\in X)$
be elements of $L$ satisfying
$\eqref{eq:shikaku1}$,
$\eqref{eq:shikaku2}$,
$\eqref{eq:shikaku3}$,
and $\eqref{eq:equivbraidcomm}$
for all $\lambda\in H, a, b, c\in L, x\in X$.
Then 
$(Y, m_Y)$ defined by $\eqref{eq:mY}$ is a left $(L, m, \eta)$-module satisfying
$\eqref{eq:mXmY}$ and that
two pairs
$(m_Y, m_Y^{\mathrm{triv}})$ and
$(m_Y, m_Y^{\sigma})$ braid-commute respectively.

\noindent{$(3)$}
The correspondence in $(1)$
is the inverse of that in $(2)$
and vice versa.
\end{corollary}

The following proposition states that
$a\shikaku{\lambda}{x}b\in L$
are recovered by
the elements $a\shikaku{\lambda}{x}e_L$.
\begin{proposition}\label{prop:equivshikakubeta}
$(1)$
If elements
$a\shikaku{\lambda}{x}b\in L$
$(\lambda\in H, a, b\in L, x\in X)$
satisfy
$\eqref{eq:shikaku1}$,
$\eqref{eq:shikaku2}$,
$\eqref{eq:shikaku3}$,
and $\eqref{eq:equivbraidcomm}$
for any $\lambda\in H, a, b, c\in L, x\in X$,
then the elements
$\beta^\lambda_x(a)=a\shikaku{\lambda}{x}e_L\in L$
$(\lambda\in H, a\in L, x\in X)$
enjoy\/$:$
\begin{align}
\label{eq:beta1}
&\beta^\lambda_x(a\underset{\lambda}{\cdot}b)=
\lambda\backslash\pi^{-1}(\pi(\lambda a)
\pi(\lambda)^{-1}\pi(\lambda\beta^\lambda_x(b))\pi(\lambda a)^{-1}
\pi(\lambda\beta^\lambda_x(a))),
\\
\label{eq:beta2}
&(\lambda b)\beta^{\lambda b}_{m_X(\lambda b)((\lambda b)
\backslash\lambda, x)}(\rho^{\lambda b}_{(\lambda b)\backslash\lambda}(a))
\\
\notag
&\quad
=
\pi^{-1}(\pi(\lambda a)\pi(\lambda)^{-1}\pi(\lambda b)
\pi(\lambda a)^{-1}\pi(\lambda\beta^\lambda_x(a)))
\end{align}
for all $\lambda\in H, a, b\in L, x\in X$.
Here, $e_L$ is the unit element of the left quasigroup $(L, \cdot)$.

\noindent{$(2)$}
We assume that elements $\beta^\lambda_x(a)\in L$ satisfy
$\eqref{eq:beta1}$
and
$\eqref{eq:beta2}$
for all $\lambda\in H, a, b\in L, x\in X$.
We write
\begin{align}
&a\shikaku{\lambda}{x}b=\lambda\backslash\pi^{-1}(\pi(\lambda a)
\pi(\lambda)^{-1}\pi(\lambda b)\pi(\lambda a)^{-1}
\pi(\lambda\beta^\lambda_x(a)))
\label{eq:shikakubeta1}
\\
&\ \;\quad(=\lambda\backslash((\lambda b)\beta^{\lambda b}_{m_X(\lambda b)
((\lambda b)\backslash\lambda, x)}
(\rho^{\lambda b}_{(\lambda b)\backslash\lambda}(a))))
\label{eq:shikakubeta2}
\end{align}
for $\lambda\in H, a, b\in L, x\in X$.
Then the elements
$a\shikaku{\lambda}{x}b$
satisfy
$\eqref{eq:shikaku1}$,
$\eqref{eq:shikaku2}$,
$\eqref{eq:shikaku3}$,
and $\eqref{eq:equivbraidcomm}$
for all $\lambda\in H, a, b, c\in L, x\in X$.

\noindent{$(3)$}
The correspondence in $(1)$
is the inverse of that in $(2)$
and vice versa.
\end{proposition}
\begin{proof}
We only prove $(1)$ and $(2)$.

\noindent{$(1)$}
Let us first show \eqref{eq:beta1}. 
Substituting $e_L$ for $c$ in \eqref{eq:equivbraidcomm}
yields
\begin{equation}
a \shikaku{\lambda}{x} b
=  \lambda \backslash \pi^{-1}(\pi(\lambda a)\pi(\lambda)^{-1}\pi(\lambda b) \pi(\lambda a)^{-1}\pi(\lambda \beta^{\lambda}_{x}(a)))
\label{eq:pfshikakuandbeta1}
\end{equation} 
for $\lambda\in H(=L), a, b\in L$,
and $x\in X$,
because
$a\underset{\lambda}{\cdot}e_L=a$
by
the definition
\eqref{eq:dotlambda}.
It follows from \eqref{eq:dotlambda} 
and
\eqref{eq:pfshikakuandbeta1}
that
\begin{align*}
&(a\underset{\lambda}{\cdot}b)  \shikaku{\lambda}{x} c \\
=& \lambda \backslash \pi^{-1}(\pi(\lambda a)\pi(\lambda)^{-1}\pi(\lambda b)
\pi(\lambda )^{-1}\pi(\lambda c)\pi(\lambda b)^{-1}\pi(\lambda )
\pi(\lambda a)^{-1}\pi(\lambda \beta^{\lambda}_{x}(a\underset{\lambda}{\cdot}b))), 
\\
& a \shikaku{\lambda}{x} (b \shikaku{\lambda}{x} c) 
\\
=& \lambda \backslash \pi^{-1}(\pi(\lambda a)\pi(\lambda)^{-1}\pi(\lambda b)
\pi(\lambda )^{-1}\pi(\lambda c)\pi(\lambda b)^{-1}
\pi(\lambda \beta^{\lambda}_{x}(b))\pi(\lambda a)^{-1}
\pi(\lambda \beta^{\lambda}_{x}(a)))
\end{align*}
for all $\lambda\in H, a, b, c\in L$, and $x\in X$. 
On account of
\eqref{eq:shikaku1}, we get \eqref{eq:beta1}.

The next task is to show \eqref{eq:beta2}. 
For the proof, 
we use the fact that \eqref{eq:shikaku3} is equivalent to the following:
for $\lambda\in H, a, b, c\in L$, and $x\in X$,
\begin{equation}
a \shikaku{\lambda}{x} (\lambda \backslash ((\lambda b) c)) 
= \lambda \backslash 
((\lambda b)(\rho^{\lambda b}_{(\lambda b) \backslash \lambda}(a) 
\shikaku{\lambda b}{m_{X}(\lambda b)((\lambda b) \backslash \lambda, x)}c)). 
\label{eq:shikaku3prime}
\end{equation}
We show
that \eqref{eq:shikaku3} induces
\eqref{eq:shikaku3prime} only.
Because of \eqref{eq:m}
and the fact that
$(X, m_X)$ is a left $(L, m, \eta)$-module,
$m_X(\lambda)(b, m_X(\lambda b)((\lambda b)\backslash\lambda, x))=x$,
and
the right-hand-side of \eqref{eq:shikaku3prime}
is consequently
$\rho^\lambda_b(\rho^{\lambda b}_{(\lambda b)\backslash\lambda}(a))
\shikaku{\lambda}{x}(\lambda\backslash((\lambda b)c))$
owing to
\eqref{eq:shikaku3}.
This element is exactly 
the left-hand-side of
\eqref{eq:shikaku3prime},
since
$\rho^\lambda_b(\rho^{\lambda b}_{(\lambda b)\backslash\lambda}(a))
=a$
by \eqref{eq:rholambda},
and this is our claim.

We now proceed the proof of \eqref{eq:beta2}. 
By substituting $e_L$ for $c$ in \eqref{eq:shikaku3prime}, 
we obtain
\begin{align}
a \shikaku{\lambda}{x} b =&  
\lambda \backslash 
((\lambda b)(\rho^{\lambda b}_{(\lambda b) \backslash \lambda}(a) 
\shikaku{\lambda b}{m_{X}(\lambda b)((\lambda b) \backslash \lambda, x)}e_{L}))
\label{eq:pfshikakuandbeta2} \\
=& \lambda \backslash 
((\lambda b) \beta^{\lambda b}_{m_{X}(\lambda b)((\lambda b) \backslash \lambda, x)}
(\rho^{\lambda b}_{(\lambda b) \backslash \lambda}(a))),
\notag
\end{align}
and combining
\eqref{eq:pfshikakuandbeta1}
and
\eqref{eq:pfshikakuandbeta2}
yields
\eqref{eq:beta2}.

\noindent{$(2)$}
A straightforward computation
with the aid of
\eqref{eq:dotlambda}
and
\eqref{eq:shikakubeta1} shows 
\eqref{eq:shikaku1} easily.

Substituting $e_L$ for $a$ and $b$
in \eqref{eq:beta1} yields
$\beta^\lambda_x(e_L)=e_L$
for
all $\lambda\in H$ and $x\in X$,
and
$\rho^{\lambda b}_{(\lambda b)\backslash\lambda}(e_L)=e_L$
by the definition
\eqref{eq:rholambda}.
Combining these,
we obtain \eqref{eq:shikaku2}.

Since $(X, m_X)$ is a left $(L, m, \eta)$-module,
$m_X((\lambda b)c)(((\lambda b)c)\backslash(\lambda b), x)
=
m_X((\lambda b)c)(((\lambda b)c)\backslash\lambda, m_X(\lambda)(b, x))$
for $\lambda \in H, b, c\in L$, and $x\in X$,
in view of \eqref{eq:m}.
In addition,
$\rho^{(\lambda b)c}_{((\lambda b)c)\backslash(\lambda b)}(a)=
\rho^{(\lambda b)c}_{((\lambda b)c)\backslash\lambda}(\rho^\lambda_b(a))$
by the definition
\eqref{eq:rholambda}.
Combining these with \eqref{eq:pfshikakuandbeta2}
yields
\eqref{eq:shikaku3}.

The task is now to show \eqref{eq:equivbraidcomm}.
From \eqref{eq:dotlambda}
and
\eqref{eq:shikakubeta1}, 
\[
a \shikaku{\lambda}{x}(b\underset{\lambda}{\cdot}c) 
= \lambda \backslash \pi^{-1}(\pi(\lambda a)\pi(\lambda)^{-1}\pi(\lambda b)
\pi(\lambda )^{-1}\pi(\lambda c)\pi(\lambda a)^{-1}
\pi(\lambda  \beta^{\lambda}_{x}(a)))
\]
for all $\lambda\in H, a, b, c\in L$, 
and $x\in X$. 
The right-hand-side of the above equation coincides with
that of \eqref{eq:equivbraidcomm},
since
\[
\pi(\lambda c)\pi(\lambda a)^{-1}\pi(\lambda \beta^{\lambda}_{x}(a)) = \pi(\lambda)\pi(\lambda a)^{-1}\pi(\lambda (a \shikaku{\lambda}{x}c))
\]
on account of 
\eqref{eq:shikakubeta1},
and $(2)$ is therefore proved.
\end{proof}
We now introduce
$\Pi^\lambda_x(a)\in L$
instead of the elements
$\beta^\lambda_x(a)\in L$.
\begin{proposition}\label{prop:equivbetaPi}
\noindent{$(1)$}
If elements
$\beta^\lambda_x(a)\in L$ 
$(\lambda\in H, a\in L, x\in X)$
satisfy
$\eqref{eq:beta1}$
and
$\eqref{eq:beta2}$
for any $\lambda\in H, a, b\in L, x\in X$,
then the elements $\Pi^\lambda_x(a)\in L$
$(\lambda\in H, a\in L, x\in X)$
defined by
\begin{equation}\label{eq:defPilambda}
\Pi^\lambda_x(a)=\lambda\backslash\pi^{-1}(\pi(\lambda)
\pi(\lambda\beta^\lambda_x(a))^{-1}\pi(\lambda a))\in L
\end{equation} 
enjoy
\begin{align}
\label{eq:Pi1}
&\Pi^\lambda_x(a\underset{\lambda}{\cdot}b)=
\Pi^\lambda_x(a)\underset{\lambda}{\cdot}\Pi^\lambda_x(b),
\\
\label{eq:Pi2}
&\Pi^\lambda_x(\lambda\backslash((\lambda a)b))
\\
=&\lambda\backslash\pi^{-1}(\pi(\lambda)\pi(\lambda a)^{-1}
 \pi((\lambda a)\Pi^{\lambda a}_{m_X(\lambda a)((\lambda a)\backslash\lambda, x)}(b))
\pi(\lambda)^{-1}\pi(\lambda\Pi^\lambda_x(a)))
\notag
\end{align}
for all $\lambda\in H, a, b\in L, x\in X$.

\noindent{$(2)$}
If elements $\Pi^\lambda_x(a)\in L$ satisfy
$\eqref{eq:Pi1}$
and
$\eqref{eq:Pi2}$
for any $\lambda\in H, a, b\in L, x\in X$,
then the elements 
\begin{equation}\label{eq:prop71111}
\beta^\lambda_x(a)=\lambda\backslash\pi^{-1}(\pi(\lambda a)
\pi(\lambda\Pi^\lambda_x(a))^{-1}\pi(\lambda))\in L
\end{equation}
enjoy
$\eqref{eq:beta1}$
and
$\eqref{eq:beta2}$
for all $\lambda\in H, a, b\in L, x\in X$.

\noindent{$(3)$}
The correspondence in $(1)$
is the inverse of that in $(2)$
and vice versa.
\end{proposition}
Obviously,
\eqref{eq:defPilambda}
induces \eqref{eq:prop71111}
and vice versa.
\begin{remark}
On account of Remark
\ref{rem:dotlambda},
\eqref{eq:Pi1} means
that
every map $\Pi^\lambda_x: L\to L$
is a homomorphism of the group
$(L, \underset{\lambda}{\cdot})$.
\end{remark}
\begin{proof}[Proof of Proposition $\ref{prop:equivbetaPi}$]
We prove $(1)$ and $(2)$ only.
The proof will be divided into a sequence of lemmas and a corollary.
\begin{lemma}\label{lemma:betaPi}
If elements $\beta^\lambda_x(a)\in L$
$(\lambda\in H, a, b\in L, x\in X)$
satisfy $\eqref{eq:beta1}$,
then the elements $\Pi^\lambda_x(a)$
defined by $\eqref{eq:defPilambda}$
enjoy
$\eqref{eq:Pi1}$.
Conversely,
if elements $\Pi^\lambda_x(a)\in L$
$(\lambda\in H, a, b\in L, x\in X)$
satisfy $\eqref{eq:Pi1}$,
then the elements $\beta^\lambda_x(a)$
defined by $\eqref{eq:prop71111}$
enjoy 
$\eqref{eq:beta1}$.
\end{lemma}
\begin{proof}
Let $\beta^\lambda_x(a)$
$(\lambda\in H, a, b\in L, x\in X)$
be elements of $L$ satisfying
$\eqref{eq:beta1}$.
By \eqref{eq:dotlambda}, \eqref{eq:beta1},
\eqref{eq:defPilambda},
and \eqref{eq:prop71111},
the left-hand-side of 
\eqref{eq:Pi1}
is
\[
\lambda\backslash\pi^{-1}(\pi(\lambda)\pi(\lambda\beta^\lambda_x(a))^{-1}
\pi(\lambda a)\pi(\lambda\beta^\lambda_x(b))^{-1}\pi(\lambda b)),
\]
which is exactly
the right-hand-side of
\eqref{eq:Pi1}
because of \eqref{eq:dotlambda}
and
\eqref{eq:prop71111}.

Conversely,
let
$\Pi^\lambda_x(a)$
$(\lambda\in H, a, b\in L, x\in X)$
be elements of $L$
satisfying $\eqref{eq:Pi1}$.
From \eqref{eq:dotlambda},
\eqref{eq:defPilambda},
\eqref{eq:Pi1},
and \eqref{eq:prop71111},
the left-hand-side of 
\eqref{eq:beta1}
is
\[
\lambda\backslash\pi^{-1}(\pi(\lambda a)\pi(\lambda)^{-1}
\pi(\lambda b)
\pi(\lambda\Pi^\lambda_x(b))^{-1}
\pi(\lambda)\pi(\lambda\Pi^\lambda_x(a))^{-1}\pi(\lambda)),
\]
which coincides with
the right-hand-side of
\eqref{eq:beta1}
due to \eqref{eq:prop71111}.
\end{proof}
By taking account of this lemma,
we are left with the task of clarifying a relation between
\eqref{eq:beta2}
and
\eqref{eq:Pi2}.
\begin{lemma}\label{lemma:PieL}
If elements $\Pi^\lambda_x(a)\in L$
$(\lambda\in H, a, b\in L, x\in X)$
satisfy $\eqref{eq:Pi1}$,
then
$\Pi^\lambda_x(e_L)=e_L$
and
\begin{equation}\label{eq:prop71112}
\pi(\lambda)\pi(\lambda\Pi^\lambda_x(a))^{-1}\pi(\lambda)
=
\pi(\lambda\Pi^\lambda_x(\lambda\backslash
\pi^{-1}(\pi(\lambda)\pi(\lambda a)^{-1}\pi(\lambda))))
\end{equation}
for all $\lambda\in H, a\in L$,
and $x\in X$.
\end{lemma}
\begin{proof}
From \eqref{eq:dotlambda},
$e_L\underset{\lambda}{\cdot}e_L=e_L$
for every $\lambda\in H(=L)$,
and
substituting
$e_L$ for
$a$ and $b$ in \eqref{eq:Pi1},
together with
\eqref{eq:dotlambda},
yields
$\Pi^\lambda_x(e_L)=e_L$
for all $\lambda\in H$
and $x\in X$
as a result.

Moreover,
by substituting
$\lambda\backslash\pi^{-1}(\pi(\lambda)\pi(\lambda a)^{-1}\pi(\lambda))$
for $b$
in
\eqref{eq:Pi1},
we can show:
the left-hand-side is
$e_L$ because
of \eqref{eq:dotlambda}
and
the fact that
$\Pi^\lambda_x(e_L)=e_L$;
and 
the right-hand-side 
is
\[
\lambda\backslash\pi^{-1}(
\pi(\lambda\Pi^\lambda_x(a))\pi(\lambda)^{-1}
\pi(\lambda\Pi^\lambda_x(\lambda\backslash
\pi^{-1}(\pi(\lambda)\pi(\lambda a)^{-1}\pi(\lambda)))))
\]
due to \eqref{eq:dotlambda}.
This gives \eqref{eq:prop71112} immediately.
\end{proof}
\begin{lemma}
Let $\beta^\lambda_x(a)$
$(\lambda\in H, a, b\in L, x\in X)$ be elements of $L$
satisfying $\eqref{eq:beta1}$,
and we define
$\Pi^\lambda_x(a)$
$(\lambda\in H, a\in L, x\in X)$
by $\eqref{eq:defPilambda}$.
Then
the following three conditions are equivalent\/$:$
\begin{enumerate}
\item[$(1)$]
$\eqref{eq:beta2}$
for
all
$\lambda\in L, a, b\in L$,
and $x\in X$\/$;$
\item[$(2)$]
$\Pi^\lambda_{m_X(\lambda)(a, x)}(b)
=
\lambda\backslash\pi^{-1}(\pi(\lambda)
\pi(\lambda a)^{-1}
\pi((\lambda a)\Pi^{\lambda a}_x(\rho^{\lambda a}_{(\lambda a)\backslash\lambda}
(b))))$
for
all
$\lambda\in L, a, b\in L$,
and $x\in X$\/$;$
\item[$(3)$]
$\eqref{eq:Pi2}$
for
all
$\lambda\in L, a, b\in L$,
and $x\in X$.
\end{enumerate}
\end{lemma}
\begin{proof}
As we have already pointed out,
\eqref{eq:defPilambda}
induces \eqref{eq:prop71111}.
From
\eqref{eq:rholambda},
\eqref{eq:prop71111},
and \eqref{eq:prop71112},
the left-hand-side of \eqref{eq:beta2} is
\[
\pi^{-1}(
\pi(\lambda a)\pi(\lambda)^{-1}
\pi((\lambda b)
\Pi^{\lambda b}_{m_X(\lambda b)((\lambda b)\backslash\lambda, x)}
((\lambda b)\backslash\pi^{-1}(\pi(\lambda)
\pi(\lambda a)^{-1}
\pi(\lambda b))))).
\]
From Lemmas \ref{lemma:betaPi}
and \ref{lemma:PieL},
\eqref{eq:prop71112} holds,
and,
by \eqref{eq:prop71111}
and 
\eqref{eq:prop71112},
the right-hand-side of \eqref{eq:beta2} is
\[
\pi^{-1}(
\pi(\lambda a)\pi(\lambda)^{-1}
\pi(\lambda b)\pi(\lambda)^{-1}
\pi(\lambda
\Pi^\lambda_x(\lambda\backslash\pi^{-1}(\pi(\lambda)\pi(\lambda a)^{-1}
\pi(\lambda))))).
\]
The condition $(1)$
is consequently equivalent to:
\begin{align}
&\pi((\lambda b)
\Pi^{\lambda b}_{m_X(\lambda b)((\lambda b)\backslash\lambda, x)}
((\lambda b)\backslash\pi^{-1}(\pi(\lambda)
\pi(\lambda a)^{-1}
\pi(\lambda b))))
\label{eq:prop71114}
\\
=&
\pi(\lambda b)\pi(\lambda)^{-1}
\pi(\lambda
\Pi^\lambda_x(\lambda\backslash\pi^{-1}(\pi(\lambda)\pi(\lambda a)^{-1}
\pi(\lambda))))
\notag
\end{align}
for all
$\lambda\in H, a, b\in L$,
and $x\in X$.

Replacing
$a$ in \eqref{eq:prop71114} by 
$\lambda\backslash\pi^{-1}(\pi(\lambda)
\pi(\lambda a)^{-1}\pi(\lambda))$,
we can assert that the above condition is equivalent to:
\begin{align*}
&\Pi^\lambda_x(a) \\
=&
\lambda\backslash\pi^{-1}(\pi(\lambda)
\pi(\lambda b)^{-1}
\pi((\lambda b)\Pi^{\lambda b}_{m_X(\lambda b)((\lambda b)\backslash\lambda, x)}
((\lambda b)\backslash\pi^{-1}(\pi(\lambda a)\pi(\lambda)^{-1}
\pi(\lambda b)))))
\end{align*}
for all
$\lambda\in H, a, b\in L$,
and $x\in X$.
In view of \eqref{eq:m},
\eqref{eq:rholambda},
and the fact that $(X, m_X)$ is a left $(L, m, \eta)$-module,
the above condition is exactly the same as the condition $(2)$,
which coincides with the condition
\begin{equation}
\label{eq:lem7161}
\pi((\lambda a)
\Pi^{\lambda a}_x(\rho^{\lambda a}_{(\lambda a)\backslash\lambda}(b)))
=
\pi(\lambda a)\pi(\lambda)^{-1}
\pi(\lambda\Pi^\lambda_{m_X(\lambda)(a, x)}(b))
\end{equation}
for all
$\lambda\in H, a, b\in L$,
and $x\in X$.

This condition is equivalent to that
\begin{equation}\label{eq:lem7141}
\pi(\lambda\Pi^\lambda_x(\rho^\lambda _a(b)))
=
\pi(\lambda)\pi(\lambda a)^{-1}
\pi((\lambda a)\Pi^{\lambda a}_{m_X(\lambda a)((\lambda a)\backslash\lambda , x)}
(b))
\end{equation}
for all
$\lambda\in H, a, b\in L$,
and $x\in X$.
In fact,
substituting respectively $\lambda a$
and
$(\lambda a)\backslash\lambda$
for
$\lambda$ and $a$ in
\eqref{eq:lem7161}
gives 
\eqref{eq:lem7141}
and vice versa.

From
\eqref{eq:dotlambda}
and
\eqref{eq:Pi1},
\begin{align}
\lambda\backslash\pi^{-1}(
\pi(\lambda\Pi^\lambda_x(\rho^\lambda_a(b)))
\pi(\lambda)^{-1}\pi(\lambda\Pi^\lambda_x(a)))
=&
\Pi^\lambda_x(\rho^\lambda_a(b))\underset{\lambda}{\cdot}
\Pi^\lambda_x(a)
\label{eq:lem7142}
\\
=&\Pi^\lambda_x(\rho^\lambda_a(b)\underset{\lambda}{\cdot}a),
\notag
\end{align}
which is exactly 
the left-hand-side of
\eqref{eq:Pi2}
in view of \eqref{eq:dotlambda} and \eqref{eq:rholambda}.

Therefore,
\eqref{eq:Pi2} holds, if and only if
the right-hand-side of \eqref{eq:Pi2}
is equal to the left-hand-side of \eqref{eq:lem7142}.
Because this equation is equivalent to
\eqref{eq:lem7141},
the conditions $(2)$ and $(3)$ are equivalent.
This is our assertion.
\end{proof}
Lemma \ref{lemma:betaPi}
immediately implies the following corollary.
\begin{corollary}
Let $\Pi^\lambda_x(a)$
$(\lambda\in H, a, b\in L, x\in X)$ be elements of $L$
satisfying $\eqref{eq:Pi1}$,
and we define
$\beta^\lambda_x(a)$
$(\lambda\in H, a\in L, x\in X)$
by $\eqref{eq:prop71111}$.
Then
the three conditions in the previous lemma are equivalent.
\end{corollary}
The above sequence of lemmas and a corollary completes the proof
of Proposition \ref{prop:equivbetaPi}
$(1)$ and $(2)$.
\end{proof}
Let $\lambda_0(\in L)$ denote the unique element satisfying
$\pi(\lambda_0)=e_G$.
Here, $e_G$ is the unit element of the group $G$.

The following proposition is our final step in the proof of Theorem
\ref{theorem:grouphomo1}.
\begin{proposition}\label{prop:equivPif}
\noindent{$(1)$}
We assume that elements $\Pi^\lambda_x(a)\in L$ 
$(\lambda\in H, a, \in L, x\in X)$
satisfy
$\eqref{eq:Pi1}$
and
$\eqref{eq:Pi2}$
for all $\lambda\in H, a, b\in L, x\in X$.
Then the maps $f^{\lambda_0}_x: G\to G$
defined by
\begin{equation}
f^{\lambda_0}_x(a)=\pi(\lambda_0\Pi^{\lambda_0}_x(\lambda_0\backslash
\pi^{-1}(a)))
\label{eq:fusingPi}
\end{equation}
are group homomorphisms
and enjoy
\begin{equation}\label{eq:Pi}
\Pi^\lambda_x(a)=\lambda\backslash\pi^{-1}(\pi(\lambda)
f^{\lambda_0}_{m_X(\lambda_0)(\lambda_0\backslash\lambda, x)}
(\pi(\lambda a))
f^{\lambda_0}_{m_X(\lambda_0)(\lambda_0\backslash\lambda, x)}
(\pi(\lambda))^{-1})
\end{equation}
for all $\lambda\in H, a\in L, x\in X$.

\noindent{$(2)$}
Let 
$\{f^{\lambda_0}_x: G\to G\mid x\in X\}$ be a family of group homomorphisms.
If we define elements $\Pi^\lambda_x(a)\in L$
$(\lambda\in H, a\in L, x\in X)$
by $\eqref{eq:Pi}$,
then they satisfy
$\eqref{eq:Pi1}$
and
$\eqref{eq:Pi2}$
for all $\lambda\in H, a, b\in L, x\in X$.

\noindent{$(3)$}
The correspondence in $(1)$
is the inverse of that in $(2)$,
and vice versa.
\end{proposition}
\begin{proof}
We only prove $(1)$ and $(2)$.

\noindent{$(1)$}
We begin by showing that  
$f^{\lambda_0}_x(ab) = f^{\lambda_0}_x(a) f^{\lambda_0}_x(b)$ 
for all $a, b \in G$. 
Because of $\eqref{eq:dotlambda}$ and the fact that $\pi(\lambda_0) = e_{G}$,
\begin{align}
f^{\lambda_0}_x(ab) =& \pi(\lambda_0\Pi^{\lambda_0}_x(\lambda_0
\backslash\pi^{-1}(a \pi(\lambda_0)^{-1} b))) 
\label{eq:pf_Piandf1}
 \\
=& \pi(\lambda_0\Pi^{\lambda_0}_x((\lambda_0 \backslash \pi^{-1}(a)) 
\underset{\lambda_0}{\cdot} (\lambda_0 \backslash \pi^{-1}(b)))). 
\notag
\end{align}
From \eqref{eq:Pi1}, the right-hand-side of \eqref{eq:pf_Piandf1} is
\begin{align*}
& \pi(\lambda_0(\Pi^{\lambda_0}_x(\lambda_0 \backslash \pi^{-1}(a)) 
\underset{\lambda_0}{\cdot} \Pi^{\lambda_0}_x(\lambda_0 \backslash \pi^{-1}(b)))) 
\\
=& \pi(\lambda_0 \Pi^{\lambda_0}_x(\lambda_0 \backslash \pi^{-1}(a))) 
\pi(\lambda_0)^{-1} \pi(\lambda_0 
\Pi^{\lambda_0}_x(\lambda_0 \backslash \pi^{-1}(b))),
\end{align*}
which coincides with $f^{\lambda_0}_x(a) f^{\lambda_0}_x(b)$, 
and the map $f^{\lambda_0}_x: G \to G$ is consequently a group homomorphism.

The next task is to show \eqref{eq:Pi}.  
From \eqref{eq:Pi2},
\begin{align}
&\Pi^{\lambda a}_{m_{X}(\lambda a)((\lambda a) \backslash \lambda, x)}(b) 
\label{eq:pf_Piandf2}
\\
=& (\lambda a) \backslash \pi^{-1}(\pi(\lambda a)\pi(\lambda)^{-1}
\pi(\lambda \Pi^{\lambda}_x(\lambda \backslash ((\lambda a) b)))
\pi(\lambda \Pi^{\lambda}_x(a))^{-1}\pi(\lambda)). 
\notag
\end{align}
Since $(X, m_{X})$ is a left $(L, m, \eta)$-module, 
$m_{X}(\lambda)(\lambda \backslash \lambda_0, 
m_{X}(\lambda_0)(\lambda_0 \backslash \lambda, x)) = x$ 
on account of \eqref{eq:m}. 
As a result, we deduce 
\eqref{eq:Pi} by means of \eqref{eq:fusingPi}, 
substituting 
$\lambda_0, 
\lambda_0\backslash\lambda, 
a, m_{X}(\lambda_0)(\lambda_0 \backslash \lambda, x)$
for $\lambda, a, b, x$ in \eqref{eq:pf_Piandf2}. 
This is our claim.

\noindent{$(2)$}
Since $(X, m_{X})$ is a left $(L, m, \eta)$-module, 
$m_{X}(\lambda_0)(\lambda_0 \backslash (\lambda a), 
m_{X}(\lambda a)((\lambda a) $ $\backslash \lambda, x)) 
= m_{X}(\lambda_0)(\lambda_0 \backslash \lambda, x)$ 
in view of \eqref{eq:m}, and we can immediately prove 
\eqref{eq:Pi2}, by using \eqref{eq:Pi}.

On account of \eqref{eq:dotlambda},  \eqref{eq:Pi}, 
and the fact that every $f^{\lambda_0}_x: G\to G$ is a group homomorphism,
we can show \eqref{eq:Pi1}, and $(2)$ is proved.
\end{proof}
Combining Corollary
\ref{cor:equivmYshikaku}
and
Propositions
\ref{prop:equivshikakubeta},
\ref{prop:equivbetaPi},
\ref{prop:equivPif},
we can prove
Theorem \ref{theorem:grouphomo1}.

We are now in a position to describe the dynamical reflection map
$k: L\otimes X\to L\otimes X\in\setH$
in \eqref{eq:k}
by means of the elements $\Pi^\lambda_x(a)$:
for $\lambda\in H, a\in L, x\in X$,
\begin{align}
&k(\lambda)(a, x)
\label{eq:REkusePi}
\\
=&(\Pi^\lambda_{m_X(\lambda)(a, x)}(a),
m_X(\lambda\Pi^\lambda_{m_X(\lambda)(a, x)}(a))
((\lambda\Pi^\lambda_{m_X(\lambda)(a, x)}(a))\backslash(\lambda a), x)).
\notag
\end{align}
\section{Examples}
\label{section:examples}
From Proposition \ref{prop:braidedmonoid},
the left quasigroup with a unit $(L, \cdot, e_L)$
(Definition \ref{def:leftquasig})
and the group $G$ isomorphic to $L$ as sets
give birth to a braided monoid $(L, m, \eta, \sigma)$ in $\setH$
(Here, $H=L$).
With the aid of Propositions
\ref{prop:k1}, \ref{prop:k2}, \ref{prop:sigmabraid},
\ref{prop:sigmaisom},
Remarks \ref{rem:braidcommmYmYtriv},
\ref{rem:braidcommmYmYsigma},
Corollary \ref{thm:reflection}, and Proposition \ref{cor:reflection},
Theorem \ref{theorem:grouphomo1} may be summarized by saying that
every left $(L, m, \eta)$-module $(X, m_X)$ in $\setH$,
together with
a family (indexed by $X$) of homomorphisms of the group $G$,
can produce a dynamical reflection map
$k: L\otimes X\to L\otimes X\in\setH$
\eqref{eq:k}.

In this section, we construct examples of dynamical reflection maps
and reflection maps
from this viewpoint
(Cf.\ \cite[Section 3]{decommer}).

We first present examples of left $(L, m, \eta)$-module $(X, m_X)$.
\begin{example}\label{ex:leftregular}
$(L, m)$ is a left $(L, m, \eta)$-module,
as is easy to check.
Here, $m: L\otimes L\to L$ is defined by
\eqref{eq:m}.
This module is called left regular.
\end{example}

The set $X$ of one element is a left $(L, m, \eta)$-module.
\begin{example}
Let $X$ be a set of one point.
Write $X=\{ x\}$.
We fix any $\lambda_1\in H$,
and define
$\lambda\cdot_Xx=\lambda_1$.
Then
$(X, \cdot_X)$ is an object of $\setH$;
moreover,
$(X, m_X)$ is a left $(L, m, \eta)$-module.
Obviously, $m_X(\lambda)(a, x)=x$
$(\lambda\in H, a\in L)$.
\end{example}

Every nonempty
set $X$ has a left $(L, m, \eta)$-module structure.
\begin{example}\label{standardexofmX}
Let $L\times X\ni (a, x)\mapsto ax\in X$
be a map satisfying
that,
for all $a\in L$ and $y\in X$,
there uniquely exists $x\in X$ such that $ax=y$.
We will use the symbol $a\backslash y$ to denote this unique element $x\in X$.
Thus, $a(a\backslash y)=y$
and
$a\backslash(ay)=y$.

For any map $f: X\to H(=L)$,
we define
$\lambda\cdot_Xx\in H$
$(\lambda\in H, x\in X)$
by
\begin{equation}
\lambda\cdot_Xx=f(e_L\backslash(\lambda x))
\ (\lambda\in H(=L), x\in X).
\label{eq:defcdotX}
\end{equation}
Then
\begin{equation}
\label{eq:cdotX}
\lambda\cdot_X(\lambda\backslash((\lambda a)x))=
(\lambda a)\cdot_Xx
\end{equation}
for all $\lambda\in H(=L), a\in L$,
and $x\in X$. 
The proof is immediate from the definition of
$\lambda\cdot_Xx\in H$.
$(X, \cdot_X)$ is hence an object of $Set_H$.
\begin{remark}\label{remstmX}
If a map $\cdot_X: L\times X\to X$
satisfies \eqref{eq:cdotX},
then
the set $\{ e_L\cdot_Xy\in L\mid
y\in X\}$ completely
determines $\lambda\cdot_Xx$
for all $\lambda\in H(=L)$ and $x\in X$.
In fact,
$\lambda\cdot_Xx=
e_L\cdot_X(e_L\backslash(\lambda x))$
from \eqref{eq:cdotX}.
\end{remark}

Let $m_X(\lambda)$
$(\lambda\in H)$ denote the map from $L\times X$ to $X$ defined by
\[
m_X(\lambda)(a, x)=\lambda\backslash((\lambda a)x) \ (a\in L, x\in X).
\]
It follows from \eqref{eq:cdotX}
that $m_X: L\otimes X\to X$ is a morphism of $\setH$.

Moreover,
$(X, m_X)$ is a left $(L, m, \eta)$-module.
In fact,
for $\lambda\in H(=L), a, b\in L$, and $x\in X$,
\begin{align*}
(m_X(m\otimes 1_X))(\lambda)((a, b), x)
=&
\lambda\backslash(((\lambda a)b)x)
\\
=&
\lambda\backslash((\lambda a)((\lambda a)\backslash(((\lambda a)b)x)))
\\
=&
(m_X(1_L\otimes m_X)a_{LLX})(\lambda)((a, b), x),
\end{align*}
and
$(m_X(\eta\otimes 1_X))(\lambda)(\bullet, x)
=
\lambda\backslash((\lambda e_L)x)
=\lambda\backslash(\lambda x)
=x
=l_X(\lambda)(\bullet, x)$.
The morphism $m_X: L\otimes X\to X$ thus satisfies
\eqref{eq:leftmod2}
and
\eqref{eq:leftmod1}.
\end{example}

Let $(X, m_X)$ be a left $(L, m, \eta)$-module,
and
let $\lambda_0(\in L)$ denote the unique element satisfying
$\pi(\lambda_0)=e_G$.
Here, $e_G$ is the unit element of the group $G$.
Next we introduce 
a family 
$\{ f^{\lambda_0}_x: G\to G \mid x\in X\}$ of homomorphisms of the group $G$.
\begin{example}\label{ex:1}
Let $f^{\lambda_0}_{x}:
G \to G$ $(x \in X)$
denote the group homomorphisms of $G$ defined by
$f^{\lambda_0}_x(a)=e_G$
$(a\in G)$.
By means of \eqref{eq:Pi},
the family
$\{ f^{\lambda_0}_x: G\to G \mid x\in X\}$ produces
$\Pi^\lambda_x(a)=e_L$ 
$(\lambda\in H(=L), a\in L, x\in X)$,
and
the dynamical reflection map
$k:L \otimes X \to L \otimes X$
in \eqref{eq:REkusePi}
is consequently
$k(\lambda)(a, x) = (e_{L}, m_{X}(\lambda)(a, x))$
(Cf.\ \cite[Example 3.1]{decommer}).
\end{example}
\begin{example}
For every $x \in X$,  we set $f^{\lambda_0}_{x}=
1_{G}$, which is a group homomorphism. 
By means of \eqref{eq:Pi},
the family
$\{ f^{\lambda_0}_x: G\to G \mid x\in X\}$ yields
\begin{equation}
\label{eq:exid}
\Pi^\lambda_x(a)=\lambda\backslash
\pi^{-1}(\pi(\lambda)\pi(\lambda a)\pi(\lambda)^{-1})
\quad(\lambda\in H(=L), a\in L, x\in X),
\end{equation}
and
the dynamical reflection map
$k:L \otimes X \to L \otimes X$
in \eqref{eq:REkusePi}
is 
\begin{align*}
&k(\lambda)(a, x) 
\\
= &(\lambda\backslash
\pi^{-1}(\pi(\lambda)\pi(\lambda a)\pi(\lambda)^{-1}),
\\
&\quad
m_X(\pi^{-1}(\pi(\lambda)\pi(\lambda a)\pi(\lambda)^{-1}))
(\pi^{-1}(\pi(\lambda)\pi(\lambda a)\pi(\lambda)^{-1})\backslash(\lambda a), x)).
\end{align*}
If the group $G$ is abelian,
then
$\Pi^\lambda_x(a)=a$ 
$(\lambda\in H(=L), a\in L, x\in X)$
from \eqref{eq:exid}
(Cf.\ \cite[Example 3.4]{decommer}),
and
the dynamical reflection map is trivial:
$k(\lambda)(a, x)=
(a, x)$.
\end{example}

If the group $G$ is abelian,
then the map
$G\ni a\mapsto a^{-1}\in G$ is a group homomorphism.
\begin{example}\label{ex:a-1}
We assume that the group $G$ is abelian. 
For every $x \in X$, 
the map $f^{\lambda_0}_{x}:
G \to G$ is defined
by $f^{\lambda_0}_{x}(a) = a^{-1}$
$(a \in G)$.
From \eqref{eq:Pi},
the family
$\{ f^{\lambda_0}_x: G\to G \mid x\in X\}$ gives birth to
$\Pi^\lambda_x(a)=\lambda\backslash
\pi^{-1}(\pi(\lambda)^2\pi(\lambda a)^{-1})$
$(\lambda\in H(=L), a\in L, x\in X)$,
and
the dynamical reflection map
$k:L \otimes X \to L \otimes X$
in \eqref{eq:REkusePi}
is 
\begin{align*}
&k(\lambda)(a, x) 
\\
= &(\lambda\backslash
\pi^{-1}(\pi(\lambda)^2\pi(\lambda a)^{-1}),
\\
&\quad
m_X(\pi^{-1}(\pi(\lambda)^2\pi(\lambda a)^{-1}))
(\pi^{-1}(\pi(\lambda)^2\pi(\lambda a)^{-1})\backslash(\lambda a), x)).
\end{align*}
\end{example}

A family of the inner automorphisms of the group $G$ can produce
a dynamical reflection map.
\begin{example}\label{ex:2}
Let $g_{x}$ $(x \in X)$
be elements of $G$, 
and let $f^{\lambda_0}_{x}:G \to G$ $(x\in X)$
denote the inner automorphisms
of the group $G$:
$f^{\lambda_0}_x(a) = g_x^{-1} a g_x$
$(a \in G)$. 
Then
the dynamical reflection map
$k:L \otimes X \to L \otimes X$ in $\eqref{eq:REkusePi}$
is 
$
k(\lambda)(a, x) 
=(\lambda\backslash b,
m_X(b)(b\backslash(\lambda a), x))$
for
$\lambda \in H, a \in L$, and $x \in X$
(Cf.\ \cite[Examples 3.6 and 3.7]{decommer}).
Here, 
\[
b=
\pi^{-1}(\pi(\lambda)(g_{m_X(\lambda_0)(\lambda_0\backslash(\lambda a), x)})^{-1}
\pi(\lambda a)\pi(\lambda)^{-1}
g_{m_X(\lambda_0)(\lambda_0\backslash(\lambda a), x)}).
\]
\end{example}

If a left $(L, m, \eta)$-module $(X, m_X)$ is left regular
(Example \ref{ex:leftregular})
and
$g_c=\pi(\lambda_0 c)$
for $c\in L(=X)$,
then
the dynamical reflection map
$k:L \otimes L \to L \otimes L$ in the above example
is exactly the morphism $\sigma^2: L\otimes L\to L\otimes L$
(Cf.\ \cite[Examples 4.3 and 8.6]{decommer}).

The following example shows that there exists a dynamical reflection map 
$k: L \otimes X \to L \otimes X$ that is dependent on $ \lambda \in H$. 
\begin{example}
\label{ex:dependent}
Let $(L, \cdot, e_{L})$ denote the left quasigroup with a unit in 
Example \ref{leftquasiexample}.
We set 
$H = L$
and $X = \{x_{1}, x_{2}, x_{3} \}$.
The map 
$L \times X \ni (a, x) \mapsto ax \in X$ is defined
by Table 2. 
We note that, for any
$a \in L$ and $ y \in X$,
there uniquely exists 
$x \in X$
such that
$ax = y$.
\begin{table}[t]
\caption{ The map $L \times X \ni (a, x) \mapsto ax \in X$}
\begin{center}
  \begin{tabular}{l | r  r  r } 
\hline
    & $x_{1}$ & $x_{2}$ & $x_{3}$  \\  \hline
   $e_{L}$ &  $x_{1}$ & $x_{2}$ & $x_{3}$ \\ 
   $l_{1}$ & $x_{2}$ & $x_{1}$ & $x_{3}$ \\  
   $l_{2}$ & $x_{3}$ & $x_{2}$ & $x_{1}$ \\ 
   $l_{3}$ & $x_{1}$ & $x_{3}$ & $x_{2}$ \\ 
   $l_{4}$ & $x_{2}$ & $x_{3}$ & $x_{1}$  \\ 
   $l_{5}$ & $x_{3}$ & $x_{1}$ & $x_{2}$  \\ \hline
  \end{tabular}
\end{center}
\end{table}

On account of Example \ref{standardexofmX}, 
we can define $\lambda\cdot_Xx \in L$
$(\lambda \in H, x \in X)$ 
by \eqref{eq:defcdotX}, 
in which the following map $f: X \to H$ is used: 
$f(x_1) = l_2; f(x_2) = l_4$; and $f(x_3) = l_3$. 
We present the map $\cdot_X: H\times X\to H$ in Table 3. 
\begin{table}[t]
\caption{ The map $\cdot_X: H \times X \to H$ }
\begin{center} 
  \begin{tabular}{l | r  r  r } 
\hline
    & $x_{1}$ & $x_{2}$ & $x_{3}$  \\ \hline 
   $e_{L}$ &  $l_{2}$ & $l_{4}$ & $l_{3}$ \\ 
   $l_{1}$ & $l_{4}$ & $l_{2}$ & $l_{3}$ \\  
   $l_{2}$ & $l_{3}$ & $l_{4}$ & $l_{2}$ \\ 
   $l_{3}$ & $l_{2}$ & $l_{3}$ & $l_{4}$ \\ 
   $l_{4}$ & $l_{4}$ & $l_{3}$ & $l_{2}$  \\ 
   $l_{5}$ & $l_{3}$ & $l_{2}$ & $l_{4}$  \\ \hline
  \end{tabular}
\end{center}
\end{table}
Example \ref{standardexofmX} can produce 
a left $(L, m, \eta)$-module $(X, m_X)$.

Let 
$G$
denote the symmetric group
$S_{3}$ of $\{1, 2, 3 \}$.
We next define a bijection $\pi: L \to G$ by: $ 
\pi(e_{L}) = \id; \pi (l_{1}) = (123); \pi(l_{2}) = (132);  \pi(l_{3}) = (12); \pi(l_{4}) = (13)$; 
and $\pi(l_{5}) = (23)$.
We write
$g_{x_1} = (132), g_{x_2} = (13)$, and $g_{x_3} = (12).$

The dynamical reflection map 
$k(\lambda): L \otimes X \to L \otimes X$ by  
Examples \ref{standardexofmX} and  \ref{ex:2} depends on 
$\lambda \in H (= L)$.
In fact,
$k(l_{1})(l_{2}, x_{2}) = (l_{2}, x_{2})$
and
$k(l_{3})(l_{2}, x_{2}) = (l_{5}, x_{1})$.
\end{example}
\begin{remark}
Each family of the group homomorphisms 
$f^{\lambda_0}_{x}:G \to G$ $(x\in X)$ in Examples
\ref{ex:1}--\ref{ex:2} is defined regularly;
however, we can choose them more variously,
because we do not need any condition among the group homomorphisms
$f^{\lambda_0}_{x}:G \to G$ $(x\in X)$.
For example, we can produce a dynamical reflection map
by means of the group homomorphisms
$f^{\lambda_0}_{x}:G \to G$ $(x\in X)$
defined by
\[
f^{\lambda_0}_{x}(a)
=
\begin{cases}
e_G,&\mbox{for some $x\in X$};
\\
a,&\mbox{otherwise}.
\end{cases}
\]
\end{remark}

At the end of this section,
we deal with the reflection map,
a dynamical reflection map 
$k(\lambda): L\times X\to L\times X$
that is independent on the dynamical parameter
$\lambda\in H$.
For this purpose, we first introduce 
a necessary and sufficient condition for
the dynamical Yang-Baxter map 
$\sigma(\lambda): L\times L\to L\times L$
\eqref{eq:sigma}
to be independent on the dynamical parameter $\lambda$.

Let $*: L \times L \to L$ denote the binary operation defined 
by 
\begin{equation}
\label{eq:group*}
a * b = \pi^{-1}(\pi(a)\pi(b))\quad (a, b \in L).
\end{equation}
That is to say, 
$(L, *)$ is a group isomorphic to the group $G$ 
via the bijection $\pi: L \to G$.  
\begin{proposition}\label{prop:indeconditionsigma}
If $(L, \cdot, e_{L})$ is a group, 
then the following two conditions are equivalent\/$:$
\begin{enumerate}
\item[$(1)$]
$\pi(e_L) = e_G$,
and the map $\sigma(\lambda)$ defined by $\eqref{eq:sigma}$
is independent on the dynamical parameter $\lambda \in H$.
Here, $e_G$ is the unit element of the group $G$\/$;$
\item[$(2)$]
$a (b * c) = (a b) * a^{-1}* (a c)$ for all $a, b, c \in L$.
Here, the element $a^{-1} (\in L)$ is the inverse of 
$a \in L$ with respect to 
the group structure $(L, *)$.
\end{enumerate}
Under the above conditions,
\begin{align}
\label{eq:YBmap}
\sigma(\lambda)(a, b)
=(a^{-1}*(ab),
\overline{a^{-1}*(ab)}ab)
\quad(a, b\in L).
\end{align}
Here, $\overline{c}(\in L)$ is the inverse of $c \in L$ 
with respect to the group structure $(L, \cdot)$.
\end{proposition}
\begin{proof}
Let us first prove the condition $(2)$ from $(1)$.
Since $\sigma(\lambda)$
is independent on the dynamical parameter $\lambda$,
$\xi_{\lambda}(a, b) = \xi_{e_L}(a, b)$ 
for all $\lambda \in H (= L)$ and $a, b \in L$. 
From \eqref{eq:mu1G}, \eqref{eq:xi}, and $\pi(e_L) = e_G$,
\begin{equation}
\overline{\lambda} \pi^{-1}(\pi(\lambda) \pi(\lambda a)^{-1} \pi((\lambda a) b)) 
= \pi^{-1}(\pi(a)^{-1} \pi( a b))
\label{eq:inde}
\end{equation}
for all $\lambda \in H$ and $a, b \in L$. 

Because of $\eqref{eq:inde}$, $\pi(e_L) = e_G$, and the fact that
\begin{equation}
\label{eq:a-1}
a^{-1} = \pi^{-1}(\pi(a)^{-1}),
\end{equation} 
\begin{align*}
 (a b) * a^{-1} * (a c) =& \pi^{-1}(\pi(ab)\pi(a)^{-1}\pi(ac)) \\
=& ab\pi^{-1}(\pi(\overline{b})^{-1}\pi(\overline{b}c)) \\
=& ab(\overline{b} \pi^{-1}(\pi(b)\pi(c))) \\
=& a (b * c)
\end{align*}
for all $a, b, c \in L$, which is the desired conclusion.

Next we show the condition $(1)$ from $(2)$. 
It follows from the condition $(2)$ that
$\pi(e_{L}) = e_{G}$.
In fact,
by substituting $e_L$ for $a$ and $c$ in the condition $(2)$,
we obtain
$b*e_L=b*e_L^{-1}*e_L=b$
for every $b\in L$,
which means that $e_L$ is the unit element
of the group $(L, *)$.
Hence,
$\pi(e_{L}) = e_{G}$.

From the condition $(2)$,
$
a\pi^{-1}(\pi(b)\pi(c)) = \pi^{-1}(\pi(ab)\pi(a)^{-1}\pi(ac))
$
for all $a, b, c \in L$.
On account of this equation, \eqref{eq:xi}, and $\pi(e_L) = e_G$, 
\begin{align}
\xi_{\lambda}(a, b) =& \overline{\lambda} \pi^{-1}(\pi(\lambda)
\pi(\lambda a)^{-1}\pi((\lambda a) b)) 
\label{eq:xiindeplambda}
\\
=& \overline{\lambda} (\lambda a\pi^{-1}(\pi(\overline{a})\pi(b))) 
\notag
\\
=& \pi^{-1}(\pi(a)^{-1}\pi(ab)),
\notag
\end{align}
and $\xi_{\lambda}(a, b)$ is thus independent on the dynamical parameter $\lambda$. 

In view of \eqref{eq:eta} and the fact that $(L, \cdot)$ is a group, 
\begin{equation}
\label{eq:indepeta}
\eta_{\lambda}(a, b) = \overline{(\lambda\xi_{\lambda}(a, b))} \lambda a b 
= \overline{\xi_{\lambda}(a, b)} a b,
\end{equation} 
and $\eta_{\lambda}(a, b)$ is also independent 
on $\lambda$. 
Therefore $(1)$ holds.

\eqref{eq:sigma},
\eqref{eq:group*},
and \eqref{eq:a-1}--\eqref{eq:indepeta}
immediately induce 
\eqref{eq:YBmap}, and the proof is complete.
\end{proof}

The set $L$ 
in Proposition \ref{prop:indeconditionsigma}
is a skew left brace \cite{decommer, guarnieri}. 
\begin{definition}\label{def:skewleftbrace}
$(L, \cdot, *)$ is a skew left brace,
iff
$(L, \cdot)$ and $(L, *)$ are groups satisfying
the condition $(2)$
in Proposition \ref{prop:indeconditionsigma}. 
\end{definition}
From this definition,
Proposition \ref{prop:indeconditionsigma} may be summarized 
by saying that the structure of the skew left brace is 
a necessary and sufficient condition 
for the dynamical Yang-Baxter map $\sigma$
\eqref{eq:sigma} to be a Yang-Baxter map 
(Cf.  \cite[Theorem 5.3]{decommer}).

In the reminder of this section
we assume that $(L, \cdot, *)$ is a skew left brace.
Then 
the map $m(\lambda): L \times L \to L$
$(\lambda \in H (= L))$ 
defined by \eqref{eq:m} is independent on $\lambda$, 
since $m(\lambda)(a, b) = \overline{\lambda} \lambda a b = ab$ for all $a, b \in L$.
As a result, the map $m(\lambda)$ is exactly a binary operation of the group 
$(L, \cdot)$; 
that is to say, the monoid $(L, m, \eta)$ is the group $(L, \cdot)$.

Let $X$ be a left $(L, \cdot)$-action. 
That is, for any $a \in L$ and $x \in X$, 
there exists $ax \in X$ satisfying $(ab)x = a(bx)$ 
and  $e_L x = x$  for all $a, b\in L$ and $x \in X$.
\begin{proposition}
For any map $f: X \to H(=L)$, we define 
$\lambda\cdot_Xx\in H$
$(\lambda\in H, x \in X)$
by
$\lambda\cdot_Xx = f(\lambda x)$.
Then 
$(X, \cdot_X)$
is an object of $Set_H$.
\end{proposition}
We set $m_X(\lambda)(a, x) = ax$
for all $\lambda \in H, a \in L$, and $x \in X$. 
Because $\lambda\cdot_Xm_X(\lambda)(a, x) 
= (\lambda a)\cdot_Xx$
$(\lambda \in H, a \in L, x \in X)$, 
$m_X: L \otimes X \to X$ is a morphism of $Set_H$, 
and moreover
\begin{proposition}
$(X, m_X)$ is a left $(L, m, \eta)$-module in $Set_H$.
\end{proposition}
That is,
every $(L, \cdot)$-action $X$, 
together with a suitable $\cdot_X: H\times X\to H$, can be regarded
as a left module of $(L, \cdot)$ in $\setH$.

From the above proposition
and $\lambda_0=e_L$, 
a left $(L, \cdot)$-action $X$, together with a family 
$\{ f^{e_L}_{x} \mid x \in X \}$
of homomorphisms of the group $G$ isomorphic to
$(L, *)$,
can produce a dynamical reflection map
$k(\lambda): L \times X \to L \times X$
$(\lambda \in H)$ 
associated with the Yang-Baxter map
$\sigma: L \times L \to L \times L$
\eqref{eq:YBmap}.

We are now in a position to show a necessary and sufficient condition 
for the dynamical reflection map $k(\lambda)$ 
to be independent on the dynamical parameter $\lambda \in H$.
\begin{proposition}\label{prop:indeconditionk}
The following two conditions are equivalent\/$:$
\begin{enumerate}
\item[$(1)$]
The dynamical reflection map $k(\lambda)$ is independent on $\lambda$\/$;$
\item[$(2)$]
$
\pi(\lambda)f^{e_L}_{\lambda (ax)}(\pi(\lambda a))
=
\pi(\lambda \pi^{-1}(f^{e_L}_{ax}(\pi(a))))
f^{e_L}_{\lambda (ax)}(\pi(\lambda))$
for all $\lambda \in H (= L), a \in L$, and $x \in X$.
\end{enumerate}
Under the above conditions,
\begin{equation}
k(\lambda)(a, x) = (\pi^{-1}(f^{e_L}_{ax}(\pi(a))),
(\overline{\pi^{-1}(f^{e_L}_{ax}(\pi(a)))}a)x)
\quad(a\in L, x\in X).
\label{eq:relationofindependentk}
\end{equation}
\end{proposition}
\begin{proof}
From \eqref{eq:REkusePi},
$k(\lambda)(a, x)
=
(\Pi^\lambda_{ax}(a),
(\overline{\Pi^\lambda_{ax}(a)}a)x)$,
and 
the condition $(1)$ is hence equivalent to
the following condition:
\begin{enumerate}
\item[$(3)$]
$\Pi^\lambda_{ax}(a)=\Pi^{e_L}_{ax}(a)$
for all $a\in L$ and $x\in X$.
\end{enumerate}
Since
$f^{e_L}_{ax}: G\to G$ is a group homomorphism,
$f^{e_L}_{ax}(e_G)=e_G$.
Because of \eqref{eq:Pi}
and
$\pi(e_L)=e_G$
(See Proposition \ref{prop:indeconditionsigma}),
the condition $(3)$ is exactly the same as
\[
\overline{\lambda}\pi^{-1}(\pi(\lambda)
f^{e_L}_{\lambda(ax)}
(\pi(\lambda a))
f^{e_L}_{\lambda(ax)}
(\pi(\lambda))^{-1})
=
\pi^{-1}(
f^{e_L}_{ax}
(\pi(a)))
\]
for all $\lambda\in H(=L), a\in L$,
and $x\in X$
(We note that $\lambda_0=e_L$),
which is equivalent to the condition $(2)$.

The above condition $(3)$,
\eqref{eq:Pi},
and
\eqref{eq:REkusePi}
immediately induce
\eqref{eq:relationofindependentk},
and
the proof is complete.
\end{proof}
\begin{remark}
On account of \eqref{eq:relationofindependentk},
this dynamical reflection map $k(\lambda): L\times X\to L\times X$ 
is exactly the same as the reflection map
in \cite[Theorem 8.5]{decommer}.
\end{remark}
By means of the family of group homomorphisms in Example \ref{ex:a-1}, 
we can produce a reflection map associated with the Yang-Baxter map
\eqref{eq:YBmap}.
\begin{example}
We assume that $(L, *)$ is abelian, and then so is the group $G$. 
The family of group homomorphisms in Example \ref{ex:a-1} 
satisfies the condition $(2)$ of Proposition \ref{prop:indeconditionk}. 

In fact, by substituting $\lambda, a$, and $a^{-1}$ 
for $a, b$, and $c$ 
in the condition $(2)$ of Proposition \ref{prop:indeconditionsigma} respectively, 
we obtain
$\lambda * (\lambda a)^{-1} = (\lambda a^{-1})* \lambda^{-1}$ 
for all $\lambda \in H (= L)$ and $a \in L$.
Here we note that $e_L (\in L)$ is the unit element of the group $(L, *)$, 
since $\pi(e_L) = e_G$. 
Because $a*b=\pi^{-1}(\pi(a)\pi(b))$
and
$\lambda^{-1}=\pi^{-1}(\pi(\lambda)^{-1})$
\eqref{eq:a-1},
it follows from the above condition that
$\pi(\lambda)\pi(\lambda a)^{-1}=\pi(\lambda\pi^{-1}(\pi(a)^{-1}))\pi(\lambda)^{-1}$
for all $\lambda \in H (= L)$ and $a \in L$,
which is exactly the same as 
the condition $(2)$ of Proposition \ref{prop:indeconditionk}.

Therefore, for every left $(L, \cdot)$-action $X$,
the dynamical reflection map
constructed here
is independent on the dynamical parameter $\lambda$,
and,
from \eqref{eq:relationofindependentk},
it is a reflection map 
$k: L\times X\ni
(a, x) \mapsto (a^{-1}, (\overline{(a^{-1})}a)x)\in
L \times X$
associated with 
the Yang-Baxter map 
\eqref{eq:YBmap}
as a result.
\end{example}
\begin{remark}
If the group $(L, *)$ is abelian,
then
$(L, \cdot, *)$ is a brace
\cite{guarnieri,rump}.
\end{remark}
\section{Quiver-theoretic solutions to reflection equation}
\label{section:quivertheoretic}
The dynamical Yang-Baxter maps can produce quiver-theoretic solutions
to the (quantum) Yang-Baxter equation,
which is due to \cite{matsumoto2}.
This section demonstrates that
the dynamical reflection maps give birth to
quiver-theoretic solutions to the reflection equation
\eqref{eq:quiverthRE}.

We first introduce the category of quivers following \cite{matsumoto2}.
Let $H$ be a nonempty set.
A quiver over $H$ is,
by definition,
a set
$Q$ with two maps $\mathfrak{s}_Q: Q\to H$
and
$\mathfrak{t}_Q: Q\to H$,
called the source map and the target map respectively.
An object of the category
$\quivH$ 
is a quiver over $H$.
For two quivers $Q$ and $Q'$ over $H$,
a morphism $f: Q\to Q'$ is
a (set-theoretic) map $f: Q\to Q'$
satisfying
$\mathfrak{s}_{Q'}(f(a))=\mathfrak{s}_Q(a)$
and
$\mathfrak{t}_{Q'}(f(a))=\mathfrak{t}_Q(a)$
for every $a\in Q$.
The identity $1_Q: Q\to Q$
is defined by the identity map $1_Q: Q\to Q$,
and
the composition of two morphisms
$f: Q\to Q', g: Q'\to Q''$
is the composition of maps $f$ and $g$:
$gf: Q\to Q''$.

This category $\quivH$ of quivers over $H$
is a tensor category (Definition \ref{def:tensorcat})
with the following fiber product.

For $Q, R\in \quivH$,
we set
$Q\times_HR=\{ (a, b)\in Q\times R\mid\mathfrak{t}_Q(a)=\mathfrak{s}_R(b)\}$.
This fiber product $Q\times_HR$ is an object of $\quivH$,
together with the source map
$\mathfrak{s}_{Q\times_HR}(a, b)=s_Q(a)$
and
the target map
$\mathfrak{t}_{Q\times_HR}(a, b)=\mathfrak{t}_R(b)$.

For two morphisms $f: Q\to Q'$
and $g: R\to R'$ of $\quivH$,
we define the fiber product $f\times_H g: Q\times_HR\to Q'\times_HR'$
by
$f\times_H g=(f\times g)|_{Q\times_H R}$.
In fact,
$(f\times g)(a, b)\in Q'\times_HR'$
for every $(a, b)\in Q\times_HR$.

The ordinary associativity of sets is the associativity constraint of
$\quivH$:
$a_{QRS}: (Q\times_HR)\times_HS\ni((a, b), c)\mapsto(a, (b, c))\in
Q\times_H(R\times_HS)$
for $Q, R, S\in\quivH$.
Obviously, this associativity constraint satisfies the pentagon axiom
\eqref{eq:pentagon}.

The set $H$ with the source and the target maps
$\mathfrak{s}_H=\mathfrak{t}_H=1_H$
is a unit of $\quivH$,
and
the left and the right unit constraints with respect to the unit $H$
are respectively defined by:
$l_Q: H\times_HQ\ni(\lambda, a)\mapsto a\in Q$
and
$r_Q: Q\times_HH\ni(a, \lambda)\mapsto a\in Q$
for every $Q\in\quivH$.

This $(\quivH, \times_H, H, a, l, r)$ is a tensor category
(Definition \ref{def:tensorcat}),
because
the triangle axiom \eqref{eq:triangle} holds.

A useful functor $Q: \setH\to\quivH$ is constructed in
\cite{matsumoto2}.
For $X\in\setH$,
we set $Q(X)=H\times X$.
This $Q(X)$,
together with the source map
$\mathfrak{s}_{Q(X)}(\lambda, x)=\lambda$
and
the target map
$\mathfrak{t}_{Q(X)}(\lambda, x)=\lambda x$,
is an object of $\quivH$.

For a morphism $f: X\to Y$ of $\setH$,
we define the map $Q(f): Q(X)\to Q(Y)$
by $Q(f)(\lambda, x)=(\lambda, f(\lambda)(x))$
$((\lambda, x)\in Q(X))$.
This $Q(f): Q(X)\to Q(Y)$ is a morphism of $\quivH$,
and $Q: \setH\to\quivH$ is hence a functor.

Moreover, this
$Q: \setH\to\quivH$ is a tensor functor
\cite[Definition XI.4.1]{kassel}.
In fact,
$\varphi_0: H\ni\lambda\mapsto(\lambda, \bullet)\in Q(I)$
is an isomorphism of $\quivH$;
and
$\varphi_2(X, Y):
Q(X)\times_HQ(Y)\ni((\lambda, x), (\kappa, y))\mapsto(\lambda, (x, y))
\in Q(X\otimes Y)$ is a natural isomorphism
for all $X, Y\in\setH$.
Since
they satisfy
\begin{align*}
&\varphi_2(X, Y\otimes Z)(1_{Q(X)}\times_H\varphi_2(Y, Z))
a_{Q(X)Q(Y)Q(Z)}
\\
&\quad=
Q(a_{XYZ})\varphi_2(X\otimes Y, Z)(\varphi_2(X, Y)\times_H1_{Q(Z)}),
\\
&l_{Q(X)}=Q(l_X)\varphi_2(I, X)(\varphi_0\times_H1_{Q(X)}),
\\
&r_{Q(X)}=Q(r_X)\varphi_2(X, I)(1_{Q(X)}\times_H\varphi_0)
\end{align*}
for all $X, Y, Z\in\setH$,
$Q: \setH\to\quivH$ is a tensor functor.
\begin{remark}
\label{rem:fullyfaithful}
The above tensor functor $Q: \setH\to\quivH$ is
fully faithful
\cite[Theorem 2.7]{matsumoto2}.
\end{remark}

Through this tensor functor $Q: \setH\to\quivH$,
every dynamical Yang-Baxter map
$\sigma: L\otimes L\to L\otimes L$
gives birth to a solution
\begin{equation}
\label{eq:quivthsol}
\tilde{\sigma}=\varphi_2(L, L)^{-1}Q(\sigma)\varphi_2(L, L):
Q(L)\times_HQ(L)\to Q(L)\times_HQ(L)
\end{equation}
to the braid relation \eqref{eq:braidrel}
in $\quivH$
\cite[Theorem 2.9]{matsumoto2}.

Let
$k: L\otimes X\to L\otimes X$
be a dynamical reflection map,
a solution to the reflection equation
\eqref{eq:RE} in $\setH$,
associated with the dynamical Yang-Baxter map
$\sigma: L\otimes L\to L\otimes L$.
The proof of
the following proposition
is straightforward.
\begin{proposition}\label{prop:quiverrefl}
The morphism
$\tilde{k}=\varphi_2(L, X)^{-1}Q(k)\varphi_2(L, X):
Q(L)\times_HQ(X)\to Q(L)\times_HQ(X)$
satisfies the following reflection equation 
in $\quivH$
associated with the solution $\tilde{\sigma}:
Q(L)\times_HQ(L)\to Q(L)\times_HQ(L)$
$\eqref{eq:quivthsol}$.
\begin{align}
&a_{Q(L)Q(L)Q(X)}^{-1}(1_{Q(L)}\times_H\tilde{k})a_{Q(L)Q(L)Q(X)}
(\tilde{\sigma}\times_H 1_{Q(X)})a_{Q(L)Q(L)Q(X)}^{-1}
\label{eq:quiverthRE}
\\
&\quad
(1_{Q(L)}\times_H\tilde{k})a_{Q(L)Q(L)Q(X)}(\tilde{\sigma}\times_H 1_{Q(X)})
a_{Q(L)Q(L)Q(X)}^{-1}
\notag
\\
=&
(\tilde{\sigma}\times_H 1_{Q(X)})a_{Q(L)Q(L)Q(X)}^{-1}
(1_{Q(L)}\times_H \tilde{k})a_{Q(L)Q(L)Q(X)}
(\tilde{\sigma}\times_H 1_{Q(X)})
\notag
\\
&\quad
a_{Q(L)Q(L)Q(X)}^{-1}
(1_{Q(L)}\times_H\tilde{k}).
\notag
\end{align}
\end{proposition}
Therefore, our method in this paper 
can produce the above
$\tilde{k}: Q(L)\times_HQ(X)\to Q(L)\times_HQ(X)$
called quiver-theoretic solutions
to the reflection equation.
\begin{remark}
The reflection equation is the defining relation of
the Artin monoid
\cite{paris}
associated with the Coxeter matrix
$\begin{pmatrix}1&4\\4&1\end{pmatrix}$.
Hence, our results in this paper may be summarized by saying that
we can construct morphisms satisfying this defining relation
in the tensor categories $\setH$ and $\quivH$.
\end{remark}
\section*{Acknowledgments}
This work was partially supported by JSPS KAKENHI Grant Numbers
JP17K05187, JP23K03062,
and Hokkaido University DX Doctoral Fellowship Grant Number 
JPMJSP2119.

\end{document}